\documentclass[11pt,oneside]{amsart}

\usepackage{amsmath,graphicx}%txfonts

\numberwithin{equation}{section}

\newtheorem{theorem}{Theorem}[section]
\newtheorem{proposition}[theorem]{Proposition}
\newtheorem{definition}[theorem]{Definition}
\newtheorem{example}[theorem]{Example}
\newtheorem{remark}[theorem]{Remark}
\newtheorem{lemma}[theorem]{Lemma}
\newtheorem{corollary}[theorem]{Corollary}
\newtheorem{algorithm}[theorem]{Algorithm}

\begin{document}

\begin{flushright}
arXiv: 1607.03569\\
\today\\
\end{flushright}

\vspace{5mm}

\begin{center}

  {\large
    Partition structure and the $A$-hypergeometric distribution\\
    associated with the rational normal curve}

  \vspace{1cm}
    
  Shuhei Mano \footnote{The Institute of Statistical Mathematics,
    10-3 Midori-cho, Tachikawa, Tokyo, 190-8562, Japan;
    E-mail: smano@ism.ac.jp}

  \vspace{1cm}
  
\end{center}

\begin{center}

{\bf Abstract}

\end{center}

A distribution whose normalization constant is an $A$-hypergeometric
polynomial is called an $A$-hypergeometric distribution. Such a distribution
is in turn a generalization of the generalized hypergeometric distribution on
the contingency tables with fixed marginal sums. In this paper, we will see
that an $A$-hypergeometric distribution with a homogeneous matrix of two
rows, especially, that associated with the rational normal curve, appears in
inferences involving exchangeable partition structures. An exact sampling
algorithm is presented for the general (any number of rows) $A$-hypergeometric
distributions. Then, the maximum likelihood estimation of
the $A$-hypergeometric distribution associated with the rational normal curve,
which is an algebraic exponential family, is discussed. The information
geometry of the Newton polytope is useful for analyzing the full and the curved
exponential family. Algebraic methods are provided for evaluating the $A$-hypergeometric polynomials.

\vspace{1cm}



\vspace{1cm}

\section{Introduction}

The $A$-hypergeometric function introduced by Gel'fand, Kapranov, and
Zelevinsky \cite{gelfand1990} is a solution of the $A$-hypergeometric system
of partial differential equations. The series solution around the origin is
called the $A$-hypergeometric polynomial.
Takayama et al. \cite{takayama2015} called a distribution whose normalization
constant is an $A$-hypergeometric polynomial as an $A$-hypergeometric
distribution. Such a distribution is in turn a generalization of
the generalized hypergeometric distribution on the contingency tables with
fixed marginal sums, and so is of interest in algebraic statistics and
information geometry. In this paper, we will see that this framework with
a homogeneous matrix $A$ of two rows helps inferences involving exchangeable
partition structures.

Exchangeable partition structures appear in count data modeling and sampling
theory, and play important roles in Bayesian statistics (see, e.g.,
\cite{charalambides2005,hjort2010,crane2016,mano2017+}). They have been
studied in the context of combinatorial stochastic processes (see , e.g.,
\cite{aldous1985,arratia2003,pitman2006}). Thanks to known results on
the $A$-hypergeometric system with a homogeneous matrix $A$ of two rows in
the contexts of commutative algebra and algebraic geometry, explicit results
can be obtained and performance of computational methods can be examined
accurately.

This paper is organized as follows. In Section~\ref{sect:Bell}, we will see
that the $A$-hypergeometric system with a homogeneous matrix $A$ of two rows is
associated with an algebraic curve known as a monomial curve. The unique
polynomial solutions of this $A$-hypergeometric system are constant multiples
of the $A$-hypergeometric polynomial. In particular, for the $A$-hypergeometric
system associated with a special monomial curve, called the rational normal
curve, the $A$-hypergeometric polynomial is a constant multiple of
the associated partial Bell polynomial, which was recently defined by
the author \cite{mano2017}.

In the following three sections, we discuss statistical applications. In
Section~\ref{sect:test}, the computational aspects of similar tests that
involve $A$-hypergeometric distributions will be discussed. As an alternative
to the Markov chain Monte Carlo with moves by a Markov basis, an exact sampling
algorithm for general (any number of rows) $A$-hypergeometric distributions
is presented. The algorithm is demonstrated in a goodness of fit test
of a Poisson regression. Section~\ref{sect:exc} sheds light on a connection
with exchangeable partition probability functions (EPPFs).
The $A$-hypergeometric distribution associated with the rational normal curve
appears as the conditional distribution of a general class of EPPFs given
the sufficient statistics. In Section~\ref{sect:MLE}, the maximum likelihood
estimation of the $A$-hypergeometric distribution will be discussed.
The information geometry of the Newton polytope of the $A$-hypergeometric
polynomial works effectively. The p.m.f. (probability mass function) is
an algebraic exponential family. From geometric properties of the Newton
polytope, we see an interesting observation (Theorem~\ref{thm:MLE_exist}):
the maximum likelihood estimator (MLE) of the full exponential family for
a count vector does not exist with probability one. So, we consider a sample
consisting of multiple count vectors and/or curved exponential families.
Gradient-based methods to evaluate the MLE will be discussed. They are
demonstrated in a problem associated with an EPPF that appears in an empirical
Bayes approach.

All the above applications demand practical methods for evaluating
the $A$-hypergeometric polynomials associated with the rational normal curve.
Section~\ref{sect:comp} is devoted to tackling this issue.
The $A$-hypergeometric polynomials satisfy a recurrence relation that comes
from the enumerative combinatorial structure of partial Bell polynomials. Use
of the recurrence relation is a method for evaluating the $A$-hypergeometric
polynomials. Lemma~\ref{lem:pfaffian} gives an explicit expression for
a system of contiguity relations among the $A$-hypergeometric polynomials,
called the Pfaffian system. By virtue of this explicit expression, alternative
algebraic methods for evaluating the $A$-hypergeometric polynomials are
presented. They are examples of methods called the holonomic gradient methods
(HGMs) \cite{nakayama2011,ohara2015,takayamaweb}. Roughly speaking,
the difference HGM demands less computational cost, while the recurrence
relation gives more accurate estimates. The performance of these methods are
compared in applications to evaluating specific $A$-hypergeometric polynomials.
If $n-k$ is large, no method is feasible and asymptotic approximations are
inevitable instead. The accuracy of known asymptotic form and that obtained by
the method developed by Takayama et al. \cite{takayama2015} are compared.

\section{Partial Bell polynomials as $A$-hypergeometric polynomials}
\label{sect:Bell}

  In this section, we will see that the unique polynomial solution of
  the $A$-hypergeometric system associated with the rational normal curve is
  a constant multiple of the associated partial Bell polynomials. The standard
  monomials for the left ideal of the $A$-hypergeometric system will be
  presented. They are useful for evaluating the $A$-hypergeometric polynomials.

Consider a partition of positive integer $n$ with $k$ positive integers:
$n=n_1+\cdots+n_k$. Here $\{n_1,...,n_k\}$ is a multiset. Support of
the p.m.f. can be represented by the set of multiplicities
$s_j:=|\{i:n_i=j\}|$, $j\in\{1,...,n\}$, that is,
\begin{equation}
  {\mathcal S}_{n,k}:=
  \left\{(s_1,...,s_n): \sum_{i=1}^n is_i=n, \sum_{i=1}^n s_i=k\right\}.
  \label{set:part}
\end{equation}
This count vector $(s_1,...,s_n)$ is the main concern of this paper. Let us
call it the size index, following a terminology introduced by Sibuya
\cite{sibuya1993}. The partial Bell polynomials are defined on the support
(\ref{set:part}) with a sequence of non-negative numbers $w_1,w_2,...$
\cite{comtet1974}:
\begin{equation}
  B_{n,k}(w):=n!\sum_{s\in {\mathcal S}_{n,k}}\prod_{i=1}^n
  \left(\frac{w_i}{i!}\right)^{s_i}\frac{1}{s_i!},\qquad n\ge k,
  \label{def:Bell}
\end{equation}
with the convention $B_{0,k}(w_\cdot)=\delta_{0,k}$. The author has defined
associated versions of the partial Bell polynomials \cite{mano2017}. They are
generalizations of the partial Bell polynomials and come from setting
restrictions on the support. The associated partial Bell polynomials are
partial Bell polynomials with some terms of the non-negative sequence set to
zero. Defining the associated versions is useful for the following discussion.

\begin{definition}[\cite{mano2017}]
  \label{defin:aBell}
  Consider a partial Bell polynomial $B_{n,k}(w)$ that is defined by
  an infinite sequence of non-negative numbers $w_1, w_2,...$. The associated
  partial Bell polynomials are defined as follows.
  \begin{eqnarray}
    B_{n,k,(r)}(w)&:=&n!\sum_{s\in{\mathcal S}_{n,k,(r)}}
    \prod_{i=1}^n\left(\frac{w_i}{i!}\right)^{s_i}\frac{1}{s_i!},
    \qquad n\ge rk,
    \label{def:aBell_d}\\
    B_{n,k}^{(r)}(w)&:=&n!\sum_{s\in{\mathcal S}_{n,k}^{(r)}}
    \prod_{i=1}^n\left(\frac{w_i}{i!}\right)^{s_i}\frac{1}{s_i!},
    \qquad k\le n\le rk,
    \label{def:aBell_u}
  \end{eqnarray}
  with the conventions $B_{n,k,(r)}(w)=0$, $n<rk$, $B^{(r)}_{n,k}(w)=0$,
  $n<k$, $n>rk$, and $B^{(r)}_{n,k}(w)=B_{n,k}(w)$, $n\le r+k-1$.
  The supports ${\mathcal S}_{n,k,(r)}$ and ${\mathcal S}_{n,k}^{(r)}$ are
  defined as
  \begin{align*}
    {\mathcal S}_{n,k,(r)}&:=
    \left\{(s_1,...,s_n): \sum_{i=r}^n is_i=n, \sum_{i=r}^n s_i=k\right\}\\
    {\mathcal S}_{n,k}^{(r)}&:=
    \left\{(s_1,...,s_n): \sum_{i=1}^r is_i=n, \sum_{i=1}^r s_i=k\right\}.
  \end{align*}
\end{definition}

The associated partial Bell polynomials (\ref{def:aBell_d}) and
(\ref{def:aBell_u}) are represented by another partial Bell polynomial or as
a linear combination of other partial Bell polynomials \cite{mano2017}. For
later discussion, we present the following fact, which was not presented in
\cite{mano2017}.

\begin{proposition}
  The associated partial Bell polynomial $(\ref{def:aBell_d})$ can be
  represented by the following partial Bell polynomial:
  \begin{equation}
    B_{n,k,(r)}(w)
    =[n]_{(r-1)k}B_{n-(r-1)k,k}
    \left(\frac{w_{\cdot+r-1}}{(\cdot+1)_{r-1}}\right),
    \qquad n\ge rk.
    \label{ide:aBell_d}
  \end{equation}
  Here, symbols for factorials $(x)_i:=x(x+1)\cdots(x+i-1)$ and
  $[x]_i:=x(x-1)\cdots(x-i+1)$ are used.
  In addition, $w_{\cdot+r-1}/(\cdot+1)_{r-1}$ means that
  the sequence $w_1,w_2,...$ in the definition of the partial Bell
  polynomial $(\ref{def:Bell})$ is replaced with $w_{i+r-1}/(i+1)_{r-1}$,
  $i\ge 1$.
\end{proposition}

\begin{proof}
  Take $(r-1)$ elements for each cluster. Then the total number of remaining
  elements is $n-(r-1)k$. The cluster sizes of the partition of the remaining
  elements into $k$ clusters are free from restrictions. Denoting
  $s_j=t_{j-r+1}$ in (\ref{def:aBell_d}), we have
  \begin{equation*}
    \frac{B_{n,k,(r)}(w)}{n!}
    =\sum_{t\in {\mathcal S}_{n-(r-1)k,k}}
    \prod_{i=r}^n
    \left(\frac{w_i}{i!}\right)^{t_{i-r+1}}\frac{1}{t_{i-r+1}!}
    =\sum_{t\in {\mathcal S}_{n-(r-1)k,k}}
    \prod_{i=1}^{n-r+1}
    \left(\frac{w_{i+r-1}}{(i+r-1)!}\right)^{s_i}\frac{1}{s_i!},
  \end{equation*}
  which is the assertion.
\end{proof}

The associated partial Bell polynomials satisfy the following recurrence
relation that comes from the enumerative combinatorial structure of
the partial Bell polynomials.

\begin{proposition}[\cite{mano2017}]
  \label{prop:aBell_u_rec}
  The partial Bell polynomials and the associated partial Bell polynomials
  $(\ref{def:aBell_u})$ satisfy  
  \begin{eqnarray*}
    B_{n+1,k}^{(r)}(w)
    =\sum_{i=0\vee(n-rk+r)}^{(r-1)\wedge(n-k+1)}
    \left(\begin{array}{c}n\\i\end{array}\right)w_{i+1}
    B_{n-i,k-1}^{(r)}(w), \qquad k\le n+1\le rk
  \end{eqnarray*}
  with $B_{i,0}^{(r)}(w)=\delta_{i,0}$, $i\in{\mathbb N}:=\{0,1,2,...\}$.
  Here, $a\vee b:=\max\{a,b\}$ and $a\wedge b:=\min\{a,b\}$. 
\end{proposition}

The \textit{Weyl algebra} of dimension $m$ is the free associative
$\mathbb{C}$-algebra
\[
  D_m=\mathbb{C}\langle x_1,...,x_m,\partial_1,...,\partial_m\rangle
\]
modulo the commutation rules
\[
  x_ix_j=x_jx_i,\,\, \partial_i\partial_j=\partial_j\partial_i,\,\,
  \partial_ix_j=x_j\partial_i\,\, {\rm for}\, i\neq j,\,\,
  {\rm and}\,\, \partial_ix_i=x_i\partial_i+1.
\]
Let $I$ be a left ideal in $D_m$. It is known that the set of standard
monomials of a Gr\"obner basis of $I$ is a basis of the factor ring
$D_m/I$, which is a vector space of $\mathbb{C}(x_1,...,x_m)$. If $I$
is a zero-dimensional ideal, $D_m/I$ is finite dimensional.
If a holomorphic function $f$ satisfies a system of differential equations
$L\bullet f=0$, $L\in I$, $f$ is called a holonomic function. Gel'fand,
Kapranov, and Zelevinsky \cite{gelfand1990} defined a class of holonomic
functions known as GKZ-hypergeometric functions, which are also referred to
as $A$-hypergeometric functions.

\begin{definition}
  Let $A$ be an integer-valued $d\times m$-matrix of rank $d$, and fix
  a vector $b\in\mathbb{C}^d$. The $A$-hypergeometric system $H_A(b)$ is
  the following system of linear partial differential equations for
  an indeterminate function $f(x)$:
  \begin{align}
    &L_i:=\sum_{j=1}^ma_{ij}\theta_j-b_i,\qquad i\in\{1,...,d\},
     \label{def:A-hyp_1}\\
    &\partial^{c^+}-\partial^{c^-},\qquad
     c^+-c^-\in {\rm ker} A\cap {\mathbb Z}^m,
     \label{def:A-hyp_2}
  \end{align}
  where $c^+_i:=c_i\vee 0$, $c^-_i:=(-c_i)\vee 0$, and
  $\theta_j:=x_j\partial_j$ $($the Euler derivative$)$. We regard $H_A(b)$ as
  a left ideal in the Weyl algebra $D_m$. We call it the $A$-hypergeometric
  ideal. The second group of annihilators generates the toric ideal $I_A$ of
  $A$.
\end{definition}

The series representation of the $A$-hypergeometric function around the origin,
namely
\begin{equation}
  Z_A(b;x):=\sum_{\{c;Ac=b,c\in{\mathbb N}^m\}}\frac{x^c}{c!}, \qquad
  x^c:=\prod_{j=1}^m x_j^{c_j},\qquad c!:=\prod_{j=1}^mc_j!,
  \label{def:A-hyp_pol}
\end{equation}
is called the $A$-hypergeometric polynomial. We set $Z_A(b;x):=0$ if
$b\notin A\cdot {\mathbb N}^m$ as the convention.

For the associated partial Bell polynomial $B_{n,k}^{(r)}(w)$, definition
(\ref{def:aBell_u}) is identical to $n!$ times the $A$-hypergeometric
polynomial with
\begin{equation}
  A=\left(
    \begin{array}{ccccc}
      0&1&2&\cdots&(r-1)\wedge (n-k)\\
      1&1&1&\cdots&1
    \end{array}
  \right),\qquad
  b=\left(\begin{array}{c}n-k\\k\end{array}\right),
 \label{def:A-hyp_Bu}
\end{equation}
and the indeterminants are identified as $x_i=w_i/i!$,
$1\le i\le r\wedge(n-k+1)$. The indeterminants will be parameters in
statistical contexts. For the associated partial Bell polynomial
$B_{n,k,(r)}(w)$, the identity is not evident. However, identity
(\ref{ide:aBell_d}) leads to an expression as the partial Bell polynomial,
which is identical to $n!$ times the $A$-hypergeometric polynomial with
\[
  A=\left(
    \begin{array}{ccccc}
      0&1&2&\cdots&n-kr\\
      1&1&1&\cdots&1    
    \end{array}
    \right),\qquad
  b=\left(\begin{array}{c}n-kr\\k\end{array}\right).
\]
The indeterminants are identified as $x_i=w_{i+r-1}/(i+r-1)!$,
$1\le i\le n-kr+1$.

In general, a homogeneous matrix of two rows generates integer partitions.
Let $0<i_1<i_2<\cdots<i_{m-1}$ be relatively prime integers (the greatest
common divisor is one). Without loss of generality, we may assume
\begin{eqnarray}
  A=\left(
    \begin{array}{ccccc}
      0&i_1&i_2&\cdots&i_{m-1}\\
      1&1&1&\cdots&1
    \end{array}
    \right), \qquad m\ge 3.
  \label{def:A-hyp_As}
\end{eqnarray}
The convex hull of the column vectors is a one-dimensional polytope, whose
volume ${\rm vol}(A)$ is $i_{m-1}$. The toric ideal $I_A$ determines a degree
$i_{m-1}$ monomial curve in the projective space ${\mathbb P}^{m-1}$.
The monomial curve is normal if and only if $i_{m-1}=m-1$. In this case,
the monomial curve is the embedding of ${\mathbb P}^1$ in ${\mathbb P}^{m-1}$
and called the rational normal curve; for background, see, e.g.,
\cite{hartshorne1977}. The indeterminants of the $A$-hypergeometric system are
identified as $x_j=w_j/j!$, $j\in \{1,i_1+1,...,i_{m-1}+1\}$, and the support
is a set of integer partitions
\begin{equation*}
  \left\{
  (s_1,...,s_n):\sum_{j=1}^{b_1+b_2}js_j=b_1+b_2, \sum_{j=1}^{b_1+b_2}
  s_j=b_2, s_l=0, l\notin\{1,i_1+1,...,i_{m-1}+1\}
  \right\},
\end{equation*}
which is not empty if and only if $b\in {\mathbb N}A$, where ${\mathbb N}A$
is the monoid spanned by the column vectors of $A$, which generate a lattice
of ${\mathbb N}^2$. In this paper, we will focus on the $A$-hypergeometric
systems associated with the rational normal curve, because they arise naturally
in statistical applications.

Theories around the $A$-hypergeometric system with a homogeneous matrix $A$
of two rows are well developed \cite{cattani1999,saito2010}. It is
straightforward to see that Lemma~1.3 of \cite{cattani1999} gives the following
fact.

\begin{lemma}
  \label{lemm:conc}
  Let $d=2$. When $b\in{\mathbb N}A$, the unique polynomial solutions of 
  the $A$-hypergeometric system $(\ref{def:A-hyp_1})$ and
  $(\ref{def:A-hyp_2})$ with the matrix $(\ref{def:A-hyp_As})$ are
  constant multiples of the $A$-hypergeometric polynomial
  $(\ref{def:A-hyp_pol})$. In particular, if the system is associated with 
  the rational normal curve, equivalently, if $i_{m-1}=m-1$ in
  $(\ref{def:A-hyp_As})$, the $A$-hypergeometric polynomial
  $(\ref{def:A-hyp_pol})$ is a constant multiple of the associated partial
  Bell polynomial defined by Definition~\ref{defin:aBell}.
\end{lemma}

If $i_{m-1}=m-1$, because ${\mathbb N}A={\mathcal C}$,
${\mathcal C}:=\left\{b\in{\mathbb N}^2:0\le b_1\le b_2(m-1)\right\}$,
the corresponding integer partition exists if and only if $b\in {\mathcal C}$.

The Buchberger algorithm and the elimination theory provide a method for
computing the reduced Gr\"obner basis of the toric ideal $I_A$
(Algorithm 4.5 of \cite{sturmfels1996}).
The minimum fiber Markov basis associated with the toric ideal $I_A$ with
the matrix $(\ref{def:A-hyp_As})$ was obtained by \cite{hara2010}.
In this paper, we will use the following minimal Gr\"obner basis.
It is streightforward to obtain the reduced Gr\"obner basis from a minimal
Markov basis.
Throughout the present paper, we fix the term order as reverse lexicographic
with $\partial_1\succ\partial_2\succ\cdots\succ\partial_{n-k+1}$.
The result is as follows.

\begin{proposition}
  \label{prop:G_A}
  A minimal Gr\"obner basis of the toric ideal $I_A$, where
  the matrix $A$ is of the form $(\ref{def:A-hyp_As})$ with
  $i_{m-1}=m-1\ge 2$, is
  \[
  G_A=
  \{\partial_i\partial_j-\partial_{i+1}\partial_{j-1};
  1\le i<j\le m, i+2\le j\}.
  \]
\end{proposition}

The standard monomials give solution bases of the $A$-hypergeometric system,
and the cardinality is called the holonomic rank. The set of standard monomials
are as follows.

\begin{proposition}
  \label{prop:stand_mono}
  For a matrix $A$ of the form $(\ref{def:A-hyp_As})$ with $i_{m-1}=m-1\ge 2$
  and any vector $b\in{\mathbb C}^2$, the totality of the standard monomials
  of the initial ideal of the $A$-hypergeometric ideal $H_A(b)$ is
  $\{1,\partial_i:3\le i\le m\}$.
\end{proposition}

\begin{proof}
  For the annihilators (\ref{def:A-hyp_1}), we have
  ${\rm in}_{\prec}(L_1)=\partial_2$ and ${\rm in}_{\prec}(L_2)=\partial_1$.
  It follows from Proposition~\ref{prop:G_A} that the initial ideal for
  the minimal Gr\"obner basis of the toric ideal $I_A$ is
  $\langle\partial_i\partial_j:2\le i\le j\le m-1\rangle$.
  Therefore, the totality of the standard monomial of the $A$-hypergeometric
  ideal $H_A(b)$ is a subset of $\{1,\partial_i:3\le i\le m\}$. However, if
  $i_m=m-1$, the holonomic rank is ${\rm rank}(H_A(b))={\rm vol(A)}=m-1$
  (see Theorem~3.7 of \cite{cattani1999} or Theorem~4.2.4 of \cite{saito2010}).
  Therefore, $\{1,\partial_i:3\le i\le m\}$ is the totality of the standard
  monomials.
\end{proof}

Before closing this section, let us see a connection between the exponential
structures in enumerative combinatorics \cite{stanley1999} and
the $A$-hypergeometric system associated with the rational normal curve.
The exponential structure is characterized by the exponential generating
function $P(z)$ of the number of possible structures $p(n)$ satisfying
\[
  P(z)=\exp(W(z)),
\]
where
\[
  P(z):=\sum_{n\ge 0}\frac{p(n)}{n!}z^n, \qquad
  W(z):=\sum_{i\ge 1}\frac{w_i}{i!}z^i,
\]
with the convention $p(0)=1$. Because $B_{n,k}(w)$ is the number of possible
structures whose number of clusters is $k$, we have
$p(n)=\sum_{k=1}^nB_{n,k}(w)$ (the Bell polynomial). Therefore,
\begin{equation}
  B_{n,k}(w)=\frac{n!}{k!}[z^n]\left\{W(z)\right\}^k.
  \label{int_A-hyp}
\end{equation}
By an argument on the de Rham cohomology, the hypergeometric ideal $H_A(b)$
eliminates the $A$-hypergeometric integral \cite{saito2010}. For $d=2$,
the integral is
\[
  \Phi_C(A,b;x):=\frac{1}{2\pi\sqrt{-1}}\int_Cf(z,x)^{b_2}z^{-b_1-1}dz, \qquad
  f(z,x):=\sum_{i=1}^mx_iz^{a_{1i}}.
\]
Taking cycle $C$ belonging to the homology group
$H_1(z\in\mathbb C\backslash\{0\}|f(z,x)\neq 0)$ gives a solution basis
(the inverse is not always true \cite{gelfand1990}). Suppose the matrix $A$ and
vector $b$ are as in (\ref{def:A-hyp_Bu}) with $r=n$. Letting $C$ be a small
cycle around the origin yields the residue of the origin:
\[
\Phi_C(A,b;x)=[z^n]\left\{\sum_{i=1}^{n-k+1}x_iz^i\right\}^k.
\]
Comparing with (\ref{int_A-hyp}) shows that this integral is a constant
multiple of the partial Bell polynomial. Other solution bases do not have such
integral representation, nevertheless, they can be obtained by perturbations of
$b$ (see Example~\ref{exa:k=n-2}).

The exponential structure is a facet of the $A$-hypergeometric system
associated with the rational normal curve. Considering the exponential
structure in the theory of the $A$-hypergeometric system provides us
broader viewpoint than that given by the enumerative combinatorics.
Section~\ref{sect:comp} will show that the framework in terms of
the $A$-hypergeometric system gives us methods for evaluating
the $A$-hypergeometric polynomials other than the method using the recurrence
relation that comes from the enumerative combinatorial structure of partial
Bell polynomials. However, formulations in terms of the $A$-hypergeometric
system sometimes involve unwanted generality. In Section~\ref{sect:comp}, we
will see that properties specific to the $A$-hypergeometric polynomial are
helpful for avoiding difficulties caused by the unwanted generality in
evaluating of the $A$-hypergeometric polynomials.

\section{Samplers for similar tests}
\label{sect:test}

  As the generalized hypergeometric distribution on the contingency table
  with fixed marginal sums, the $A$-hypergeometric distribution appears as
  a conditional distribution of some model with given sufficient statistics.
  For such a case, a similar test can be conducted with the conditional
  distribution \cite{lehmann2005} with the aid of samplers from the conditional
  distribution. Constructing a sampler with algebraic constraints has been one
  of the motivating problems in algebraic statistics to date
  \cite{diaconis1998a} (recent developments in this line of research can be
  found in \cite{aoki2012}). In this section, we will discuss the computational
  aspects of samplers for the $A$-hypergeometric distribution. An exact
  sampling algorithm is proposed for general (any number of rows)
  $A$-hypergeometric distributions. Then an application to
  the $A$-hypergeometric distribution associated with the rational normal
  curve is presented.

Suppose we have a model whose conditional distribution given the sufficient
statistics is the $A$-hypergeometric distribution, namely
\begin{equation}
  q(c;x)=\frac{1}{Z_A(b;x)}\frac{x^c}{c!}, \qquad x\in \mathbb{R}_{>0}^m,
  \label{def:A_dist_sim}
\end{equation}
where $A$ is an integer valued $d\times m$-matrix of rank $d$,
$b\in\mathbb{C}^d$, and $c\in\mathbb{N}^m$ is a count vector of $m$
categories with $c_1+\cdots+c_m=k$. For a similar test of hypothesis
$H_0:x=x_0$, consider using the probability function $q(c;x)$ as the test
statistic. The significance probability of the observation $c_{obs}$ is
\begin{equation}
  {\mathbb P}(q(C;x_0)<q(c_{obs};x_0)),
  \label{def:sig_prob}
\end{equation}
where $C$ follows the $A$-hypergeometric distribution with parameter $x_0$.
To estimate the significance probability, we need an unbiased sampler from
the $A$-hypergeometric distribution.

In a Markov chain Monte Carlo (MCMC), the state space of the irreducible
Markov chain is represented as a $b$-fiber: ${\mathcal F}_b(A):=\{c:Ac=b\}$.
The set ${\mathcal M}(A)={\rm Ker}(A)\cap {\mathbb Z}^n$ is called the moves of
$A$. Consider the decomposition of ${\mathcal F}_b(A)$ into equivalence classes
induced by the connectivity with respect to
${\mathcal B}\subset{\mathcal M}(A)$. If ${\mathcal F}_b(A)$ forms one
equivalence class for all $b$, ${\mathcal B}$ is called a Markov basis
\cite{diaconis1998a}. The move has one-to-one correspondence to the binomial
ideal of polynomial ring ${\mathbb C}[x]$: $z\mapsto x^{z^+}-x^{z^-}$.

In uses of an MCMC sampler, we must assess the convergence of the chain to
the target distribution to guarantee that a sample is taken from the target
distribution. However, such assessment is not always easy.
In contrast to MCMC, the following algorithm can sample from the target
distribution exactly. The cost we must pay is to evaluate
the $A$-hypergeometric polynomials. This type of algorithm was proposed for
a test that appeared in genetics and that involves an exchangeable partition
probability function \cite{stewart1977}. However, the following algorithm can
apply to general $A$-hypergeometric distributions.

\begin{algorithm}\label{alg:exact}\rm
  A count vector $c$ with $c_1+\cdots+c_m=k$ is sampled from
  the $A$-hypergeometric distribution $(\ref{def:A_dist_sim})$ by
  the following steps. Let $I_i\in\{1,...,m\}$ be the indicator of
  the category of the $i$-th observation of the sample of size $k$,
  where $c_i=|\{j:I_j=i\}|$.
  
  \begin{itemize}

  \item[(1)]
    $I_1\sim {\mathbb P}(I_1=j)$;

  \item[(2)] For $l=2,...,k$,
    $I_l\sim {\mathbb P}(I_l=j|i_1,i_2,...,i_{l-1})$.
    
  \end{itemize}
  Here, if $b\ge a_{i_1}+\cdots+a_{i_{l-1}}+a_{i_j}$,
  \[
    {\mathbb P}(I_l=j|i_1,...,i_{l-1})=
    \frac{Z_A(b-a_{i_1}-\cdots-a_{i_{l-1}}-a_j;x)}
         {Z_A(b-a_{i_1}-\cdots-a_{i_{l-1}};x)}\frac{x_j}{k-l+1},
   \]
   else ${\mathbb P}(I_l=j|i_1,...,i_{l-1})=0$, where $a_i$ is the $i$-th
   column vector of the matrix $A$.
\end{algorithm}

Let us consider an application of Algorithm~\ref{alg:exact} to
the $A$-hypergeometric distribution associated with the rational normal curve.
The p.m.f. is
\begin{equation}
  q_{n,k}(s;x)=\frac{1}{Z_A((n-k,k)^\top;x)}\frac{x^s}{s!}, \qquad
  x\in\mathbb{R}^{r\wedge(n-k+1)}_{>0},
  \label{def:A_dist_sim2}
\end{equation}
where the matrix $A$ is given in (\ref{def:A-hyp_Bu}).
Now the count vector $c$ is the size index $s$, and we put $m=r\wedge(n-k+1)$.
We assume $m\ge 3$ and $k\ge 2$, since otherwise the sampling is trivial.
A Markov basis and the Metropolis-Hastings ratio are as follows.

\begin{proposition}
  \label{prop:MCMC}
  For the toric ideal $I_A$, where the matrix $A$ is of the form
  $(\ref{def:A-hyp_As})$ with $i_{m-1}=m-1\ge 2$, a Markov basis is
  \[
    {\mathcal B}=\{e_i+e_j-e_{i+1}-e_{j-1}; 1\le i<j\le m, 
    i+2\le j\}.
  \]
  For $j>i+2$, the Metropolis-Hastings ratio for the move from state $s$
  to state $s+\epsilon z$ is
  \[
    \frac{q_{n,k}(s+\epsilon z;x)}{q_{n,k}(s;x)}
    =\left\{
    \frac{x_i}{x_{i+1}}
    \frac{x_j}{x_{j-1}}\right\}^\epsilon\\
    \times
    \left\{
    \begin{array}{ll}
    s_{i+1}s_{j-1}\{(s_i+1)(s_j+1)\}^{-1}, &\epsilon=+1,\\ 
    s_is_j\{(s_{i+1}+1)(s_{j-1}+1)\}^{-1}, &\epsilon=-1,
    \end{array}\right.
  \]
  where $z=e_i+e_j-e_{i+1}-e_{j-1}$. For $j=i+2$,
  \[
    \frac{q_{n,k}(s+\epsilon z;x)}{q_{n,k}(s;x)}
    =\left\{\frac{x_ix_{i+2}}{x_{i+1}^2}\right\}^\epsilon\\
    \times
    \left\{
    \begin{array}{ll}
    s_{i+1}(s_{i+1}-1)\{(s_i+1)(s_{i+2}+1)\}^{-1}, &\epsilon=+1,\\ 
    s_is_{i+2}\{(s_{i+1}+2)(s_{i+1}+1)\}^{-1}, &\epsilon=-1,
    \end{array}\right.
  \]
  where $z=e_i+e_{i+2}-2e_{i+1}$.
\end{proposition}

\begin{proof}
  By virtue of Theorem~3.1 of \cite{diaconis1998a} the minimal
  Gr\"obner basis $G_A$ of the toric ideal $I_A$, which is given by
  Proposition~\ref{prop:G_A} while replacing $\partial_i$ by $x_i$, is
  a Markov basis. The Metropolis-Hastings ratio follows by a simple
  calculation.
\end{proof}

\begin{example}\rm
  The data set considered is from \cite{diaconis1998b} and concerns a goodness
  of fit of a regression model to effect of an insecticide. Consider
  the univariate Poisson regression with $m$ levels of a covariate. The means
  that $\mu_i$, $i\in\{1,...,m\}$ of independent Poisson random variables
  $S_i$ were modeled as $\log\mu_i=\alpha+\beta i$. The sufficient statistics
  are the sample size $k=\sum_{i=1}^m s_i$ and the sum of the levels
  $n=\sum_{i=1}^m is_i$. The conditional distribution given the sufficient
  statistics is the $A$-hypergeometric distribution (\ref{def:A_dist_sim2})
  with $r=m$ and $x_i=1/i!$, $i\ge 1.$ A chemical to control insects is
  sprayed on successive equally infested plots in increasing concentrations
  $1,2,3,4,5$ (in some units). After the spraying the number of insects left
  alive on the plots are $(s_1,s_2,s_3,s_4,s_5)=(44,25,21,19,11)$. $k=120$ and
  $n=288$. The similar test tells us how well the model fit to the data.
  An estimate of the significance probability of
  the $\chi^2$-statistic based on $900,000$ samples from the exact sampler
  (Algorithm~\ref{alg:exact}) was $0.0258$. To evaluate the $A$-hypergeometric
  polynomials, the recurrence relation in Proposition~\ref{prop:aBell_u_rec}
  was employed. This should be close to the true value. For the MCMC,
  an estimate based on a walk of $90,000$ steps (with the initial $10,000$
  steps having been discarded to avoid sampling from the un-converged part of
  the chain, as was done in \cite{diaconis1998b}) was $0.0231$. We can say
  that the MCMC sampling scheme gives a reasonable estimate. This can be
  confirmed with the histograms shown
  in Figure 1. The histogram obtained by the MCMC sampler is fairly close to
  that obtained by the exact sampler.
\end{example}

\section{Exchangeable partition probability functions}
\label{sect:exc}

  Chapter 1 of \cite{pitman2006} is an extensive survey of the relationship
  between the partial Bell polynomials and the exchangeable partition
  probability functions (EPPFs). A typical application of EPPFs is in
  Bayesian statistics. For a multinomial sampling from a prior distribution,
  the marginal likelihood of a sample is an EPPF (see
  Example~\ref{exa:dir_multi}). In the context of Bayesian nonparametrics,
  a prior process characterized by an EPPF is called a species sampling prior
  \cite{hjort2010,lee2013}. In this section, we will see that the conditional
  distribution of a general class of EPPFs is the $A$-hypergeometric
  distribution associated with the rational normal curve.

Label each observation of a sample of size $n$ with a positive integer, and
consider a probability law on partitions of the set $\{1,2,...,n\}$. If we
assume exchangeability, then cluster sizes are our concern. Hence, we consider
a probability law on a set of integers whose sum is a positive integer $n$.
Following Aldous \cite{aldous1985}, let us call such a probability law
a random partition. We say that a random partition $\Pi_n$ is exchangeable if
a symmetric function $p_n$ on a set of partitions of an integer $n$ satisfies
\[
  {\mathbb P}(\Pi_n=\{A_1,...,A_k\})=p_n(|A_1|,...,|A_k|)
\]
for a partition of $\{1,2,...,n\}$ to be arbitrary $k$ clusters
$\{A_1,...,A_k\}$. This p.m.f. $p_n$ is called an EPPF.

Let us consider a class of EPPFs that have a multiplicative form, namely
\[
  p_n(n_1,...,n_k)=v_{n,k}\prod_{i=1}^k w_{n_i}. \qquad
  n=n_1+\cdots+n_k
\]
The support is given by partitions of a fixed positive integer $n$ with $k$
positive integers. The parameters are two sequences of positive numbers
$(v_{n,k})$ and $(w_i)$, $1\le i,k\le n$. This EPPF is an example of
multiplicative measures, which were studied by Vershik \cite{vershik1996} as
a model of statistical mechanics. Here, a cluster of size $i$ has $w_i$
different microscopic structures. In terms of the size index, we have
\begin{equation}
  {\mathbb P}(S=s)=v_{n,k}n!\frac{x^s}{s!}, \qquad s\in{\mathcal S}_{n,k},
  \label{def:Gibbs}
\end{equation}
where $x_i=w_i/i!$, $i\in\{1,...,n-k+1\}$ and the support ${\mathcal S}_{n,k}$
is given in (\ref{set:part}). The number of clusters $|\Pi_n|$ is
the sufficient statistic for $v$ and is distributed as
${\mathbb P}(|\Pi_n|=k)=v_{n,k}B_{n,k}(w)$, where $B_{n,k}(w)$ is the partial
Bell polynomial. The conditional distribution is
\begin{equation}
  {\mathbb P}(S=s|\Pi_n|=k)=\frac{n!}{B_{n,k}(w)}\frac{x^s}{s!}, \qquad
  x\in{\mathbb R}^{n-k+1}_{>0}.
  \label{def:mGibbs}
\end{equation}
In \cite{pitman2006,vershik1996}, this  p.m.f. was referred to as
the microcanonical Gibbs distribution. If we consider logarithms of $x$ are
natural parameters, this is an exponential family. Moreover, this is
the $A$-hypergeometric distribution, since the partial Bell polynomial is $n!$
times the $A$-hypergeometric polynomial associated with the rational normal
curve, where the matrix $A$ and vector $b$ are given in (\ref{def:A-hyp_Bu})
with $r=n$. Let $Z_{n,k}(x)\equiv Z_A((n-k,k)^\top;x)=B_{n,k}(w)/n!$.

\begin{lemma}\label{lem:suffic}
  For the distribution $(\ref{def:Gibbs})$ with two sequences of positive
  numbers $v$ and $x$, the number of clusters is a sufficient and complete
  statistic for the parameter $v$.
\end{lemma}

\begin{proof}
  Sufficiency is obvious by the factorization theorem. For completeness,
  assume for a function $f(\cdot)$ that the number of clusters $|\Pi_n|$
  satisfies ${\mathbb E}(f(|\Pi_n|))=0$ for arbitrary $v$. Choose arbitrary
  $k_0$ in $\{1,...,n\}$ and fix the parameter as
  $v_{n,k}=\delta_{k,k_0}(Z_{n,k_0}(w))^{-1}$. Then we have
  ${\mathbb E}(f(|\Pi_n|))=f(k_0)=0$. This implies $f(k)\equiv 0$,
  $\forall k$, which is completeness.
\end{proof}

The following proposition is a generalization of Theorem 2.5 of
\cite{keener1978}.

\begin{proposition}
  For the distribution $(\ref{def:Gibbs})$ with two sequences of positive
  numbers $v$ and known $x$, the unique minimum variance unbiased estimator
  $(UMVUE)$ of moments of the joint factorial moments of the size index is
  \[
    {\mathbb E}
    \left[\prod_{i=1}^n[S_i]_{r_i}||\Pi_n|=k\right]=
    \frac{Z_{n-i_1r_1-\cdots-i_nr_n,k-r_1-\cdots-r_n}(x)}
         {Z_{n,k}(x)}x^rI_{\{n-k\ge(i_1-1)r_1+\cdots+(i_n-1)r_n\}}.
  \]
\end{proposition}

\begin{proof}
  The conditional distribution (\ref{def:mGibbs}) yields
  \begin{eqnarray*}
    {\mathbb E}([S_i]_{r_i}||\Pi_n|=k)
    &=&\frac{1}{Z_{n,k}(x)}
    \sum_{s\in {\mathcal S}_{n,k}}\frac{x^s}{s!}[s_i]_{r_i}
    =\frac{1}{Z_{n,k}(x)}\sum_{s'\in {\mathcal S}_{n-ir_i,k-r_i}}
    \frac{x^{s'}}{s'!}x_i^{r_i}\\
    &=&\frac{Z_{n-ir_i,k-r_i}(x)}{Z_{n,k}(x)}x_i^{r_i},
  \end{eqnarray*}
  where the multiplicity vector $s'$ is $s'_j=s_j$, $j\neq i$ and
  $s'_i=s_i-r_i$. The joint moments are derived in the same manner.
  The Lehmann-Scheff\'e theorem \cite{lehmann1998} gives the assertion.
\end{proof}

The sequences $v$ and $x$ may be parametrized by a few parameters.
An important parametrization of $x$ is
\begin{equation}
  x_i=\frac{(1-\alpha)_{i-1}}{i!}, \qquad i =1,2,..., \qquad
  -\infty<\alpha<1.
  \label{def:egibbs}
\end{equation}
Gnedin and Pitman \cite{gnedin2005} showed that an EPPF has infinite
exchangeability if and only if $x$ has this parametrization. Such
a multiplicative measure is called the Gibbs random partition.
The Gibbs random partition characterizes an important class of prior processes
in Bayesian nonparametrics \cite{hjort2010}. The Gibbs random partition is
the marginal likelihood of a sample taken from the prior process (see
Example~\ref{exa:dir_multi}). The two-parameter Dirichlet process, which is
also called the Pitman-Yor process \cite{pitman1997,ishwaran2003}, is
a popular prior process in Bayesian nonparametrics \cite{hjort2010}.
The Pitman random partition \cite{pitman1995} is a member of the Gibbs random
partitions, which is the marginal likelihood for the two-parameter Dirichlet
process.
For nonzero $\alpha$ the partial Bell polynomial has the form
\[
B_{n,k}(w)=n!Z_{n,k}(x)=\frac{(-1)^n}{(-\alpha)^k}C(n,k;\alpha).
\]
Here, $C(n,k;\alpha)$ is the generalized factorial coefficient, which
satisfies
\[
  \sum_{k=0}^nC(n,k;\alpha)[x]_k=[\alpha x]_n.
\]
For $\alpha=0$, the partial Bell polynomial is the unsigned Stirling number
of the first kind.

\begin{example}\rm
  \label{exa:dir_multi}
  Estimating the number of unseen species is an intriguing classical problem
  (recent progress can be found in, for example, \cite{lijoi2007,sibuya2014}).
  An empirical Bayes approach is as follows \cite{keener1978}. Suppose
  the frequencies of species in a population follow the $m$-variate symmetric
  Dirichlet distribution of parameter $(-\alpha)$, $\alpha<0$, where the total
  number of species is $m$. For the multinomial sampling of size $n$ with 
  the number of individuals of the $i$-th species is $n_i$, the marginal
  likelihood becomes
  \[
    p_n(n_1,...,n_m)=\left(
      \begin{array}{c}
        m\alpha\\n
      \end{array}
    \right)^{-1}
    \prod_{i=1}^m
    \left(
      \begin{array}{c}
        \alpha\\n_i
      \end{array}
    \right).
  \]
  Here, $k:=|\{i:n_i>0\}|$ is the number of observed species. This EPPF is
  the Dirichlet-multinomial or the negative hypergeometric distribution. In
  terms of the size index, we have
  \[
    \mathbb{P}(S=s)
    =\frac{[m]_k(-\alpha)^k}{(-m\alpha)_n}n!\frac{x^s}{s!}, \qquad
    x_i=\frac{(1-\alpha)_{i-1}}{i!}, \qquad i\ge 1.
  \]
  As this expression shows, the Dirichlet-multinomial distribution is
  an example of a Gibbs random partition. The number of observed species $k$ is
  the sufficient statistic of the total number of species $m$. If the parameter
  $\alpha$ is known, the UMVUE of $m$ is $\hat{m}(k)=k-\alpha^{-1}Z_{n,k-1}((1-\alpha)_{\cdot-1}/\cdot!)/Z_{n,k}((1-\alpha)_{\cdot-1}/\cdot!)$, where
  $(1-\alpha)_{\cdot-1}/\cdot!$ represents the sequence
  $x_i=(1-\alpha)_{i-1}/i!$, $i\ge 1$. Applications to some data sets can be
  found in \cite{keener1978}.
\end{example}

\section{Maximum Likelihood Estimation}
\label{sect:MLE}

In this section, we will discuss the maximum likelihood estimation of
the $A$-hypergeometric distribution associated with the rational normal curve.
Takayama, et. al~\cite{takayama2015} gave a framework for the general
$A$-hypergeometric distributions, while this section presents some results on
the $A$-hypergeometric distribution associated with the rational normal curve.
The main tools employed here are the same as those employed in
\cite{takayama2015}, but more detailed analyses are possible thanks to
specific properties of the $A$-hypergeometric system associated with
the rational normal curve, such as the relationship with the partition
polytopes. The information geometry of the Newton polytope of
the $A$-hypergeometric polynomial plays important roles throughout this
section. The p.m.f. is an algebraic exponential family. The maximum likelihood
estimation of the full and curved exponential families is discussed.
Gradient-based methods to evaluate the maximum likelihood estimator (MLE) will
be discussed. An application to a problem associated with an EPPF that appears
in an empirical Bayes approach is then presented.

Let us consider a particular $A$-hypergeometric distribution associated with
the rational normal curve, whose p.m.f. is
\begin{equation}
  q_{n,k}(s;x)=\frac{1}{Z_A((n-k,k)^\top;x)}\frac{x^s}{s!}, \qquad
  x\in{\mathbb R}^{n-k+1}_{>0},
  \label{def:A_dist_MLE}
\end{equation}
where the matrix $A$ is given in (\ref{def:A-hyp_Bu}) with $r=n\ge k+2\ge 4$
and the support is (\ref{set:part}). With this setting the $A$-hypergeometric
polynomial is $1/n!$ times the partial Bell polynomial. Although this
setting makes the discussion model-specific, the model covers important
statistical applications. As in the previous section, let
$Z_{n,k}(x)\equiv Z_A((n-k,k)^\top;x)$. The p.m.f. is the exponential family
and the log likelihood is
\[
  \ell_{n,k}(s;\xi):=\xi^is_i-\psi_{n,k}(\xi), \qquad
  \xi^i=\log x_i\in {\mathbb R},
  \qquad i\in\{1,...,n-k+1\},
\]
where $\psi_{n,k}(\xi):=\log Z_{n,k}(e^\xi)$ is the potential, and a constant
is omitted. Here and in the following Einstein's summation convention will be
used; indices denoted by a repeated letter, where the one appears as
a superscript while the other appears as a subscript, are summed up.
The superscripts should not be confused with a power. The p.m.f.
(\ref{def:A_dist_MLE}) is regular because the natural parameter space
$\{\xi: Z_{n,k}(e^\xi)<\infty\}$ is $\mathbb{R}^{n-k+1}$
\cite{brown1986,barndorff-nielsen2014}. Moreover, the likelihood is
an algebraic exponential family as defined by Drton and Sullivant
\cite{drton2007}, since the moments of the constraints in (\ref{set:part}):
$\eta_1+\cdots+\eta_{n-k+1}-k=0$ and
$\eta_1+2\eta_2+\cdots+(n-k+1)\eta_{n-k+1}-n=0$ are algebraic (polynomial)
constraints.

Under a transformation of the indeterminants the $A$-hypergeometric polynomial
transforms as
\[
Z_A(b;s_1^{\cdot-1}s_2 x_\cdot)=s_1^{n-k}s_2^kZ_A(b;x).
\]
This transformation is known as the torus action, namely
\begin{equation}
  x_i\mapsto x_is^{a_i}, \qquad i\in\{1,...,n-k+1\},
  \label{def:torus}
\end{equation}
where $a_i$ is the $i$-th column vector of the matrix $A$. This is a known
property of partial Bell polynomials~\cite{comtet1974}. Following Takayama et
al. \cite{takayama2015}, let us introduce the generalized odds
ratio to parametrize the quotient space
${\mathbb R}^{n-k+1}_{>0}/{\rm Im} A^\top$. The Gale transform of $A$, which
will be denoted as $\bar{A}$, satisfies $\bar{A}A^\top=0$. The explicit forms
of the row vectors are
\begin{equation*}
  \bar{a}_i=ie_1-(i+1)e_2+e_{i+2}, \qquad i\in \{1,...,n-k-1\},
\end{equation*}
where $e_i$ is the $(n-k+1)$-dimensional unit vector whose $i$-th component
is unity. The Gale transformation provides the generalized odds ratios, namely
\begin{equation}
  y_i:=x^{\bar{a}_j}=\frac{x_1^ix_{i+2}}{x_2^{i+1}}, \qquad
  i\in\{1,...,n-k-1\}.
  \label{def:odds}
\end{equation}

The moment map is invariant under the torus action. It can be seen that
the moment map
${\mathbb E}(S):{\mathbb R}^{n-k+1}/{\rm Im}A^\top \ni \log y\mapsto \eta$
provides the dual ($\eta$-) coordinate system in the sense of information
geometry. The dual coordinate and the Fisher metric are immediately given as 
\[
  \eta_i:=\theta_i\psi_{n,k}={\mathbb E}[S_i]
  =\frac{Z_{n-i,k-1}(y)}{Z_{n,k}(y)}y_{i-2}
%  \label{def:dual_coord}
\]
and 
\begin{equation}
  g_{ij}:=\partial_i\partial_j\psi_{n,k}={\rm Cov}[S_i,S_j]
  =\frac{Z_{n-i-j,k-2}(y)}{Z_{n,k}(y)}y_{i-2}y_{j-2}
    I_{\{n-k+2\ge i+j\}}-\eta_i\eta_j+\eta_i\delta_{i,j},
  \label{def:metric}
\end{equation}
respectively, where $\partial_i:=\partial/\partial\xi^i=\theta_i$.
Here, the torus action with $s_1=x_2^{-1}x_1$ and $s_2=x_1^{-1}$ in
(\ref{def:torus}) is used such that the vector $y$ becomes
$(1,1,y_1,...,y_{n-k-1})$. $y_{-1}=y_0=1$. Because
of the dually flatness, an exponential family is e-flat and also
m-flat~\cite{amari2000}.

Because the gradient of the log likelihood is
$\partial_i\ell_{n,k}(s;\xi)=s_i-\eta_i$, finding the MLE is equivalent to
finding the inverse image of the moment map
\begin{equation}
  {\mathbb E}(S) : {\mathbb R}^{n-k+1}/{\rm Im}A^\top\to
  {\rm relint}({\rm New}(Z_{n,k})),
  \label{def:mom_map}
\end{equation}
where the Newton polytope ${\rm New}(Z_{n,k})$ is the convex hull of
the support ${\mathcal S_{n,k}}$ given in (\ref{set:part}). The following
theorem comes from the fact that a size index never enters the right-hand side
of (\ref{def:mom_map}).

\begin{theorem}
  \label{thm:MLE_exist}
  For the likelihood given by the $A$-hypergeometric distribution associated
  with the rational curve $(\ref{def:A_dist_MLE})$, the MLE does not exist with
  probability one.
\end{theorem}

\begin{remark}\rm
  This assertion might seem curious, but we have an analogy in the theory of
  exponential families as follows. If the sample size is one, MLE of the beta
  distribution and that of the gamma distribution do not exist with probability
  one. This is because the sufficient statistics are on the boundary of
  the parameter space (see Example 5.6 in \cite{brown1986} and Example 9.8 in
  \cite{barndorff-nielsen2014}). For the $A$-hypergeometric distribution,
  the size index, a count vector of multiplicities, can be regarded as
  a multivariate sample of size one. A similar argument appeared in the context
  of algebraic statistics on a hierarchical log-linear models
  \cite{eriksson2006}.
\end{remark}

In the following discussion, the partition polytope is useful. Denote the set
of possible partitions of positive integer $n$ by
${\mathcal S}_n:=\cup_{i=1}^{n}{\mathcal S}_{n,i}$. The convex hull of
${\mathcal S}_n$ is called the partition polytope $P_n$, which was discussed
by \cite{shlyk2005}. The partition polytope has an important property, namely
that $P_n$ is a \textit{pyramid} with the apex $e_{n}$. In other words, all
vertices are on the faces of $P_n$ because the $n$-th coordinate of the apex
is 1 and that of the other vertices are zero.

\begin{proof}
  Because the p.m.f. (\ref{def:A_dist_MLE}) is regular, the MLE exists if and
  only if the sufficient statistics are in the interior of the convex hull of
  the support (Theorem 5.5 of \cite{brown1986} and Corollary 9.6 of
  \cite{barndorff-nielsen2014}). The condition is
  $s\in{\rm relint}({\rm New}(Z_{n,k}))$. If $n\ge k\ge n/2$, there is
  a one-to-one affine map between the vertices in ${\mathcal S}_{n-k}$ and
  those in ${\mathcal S}_{n,k}$:
  \begin{equation}
    {\mathcal S}_{n-k}\ni (s_1,...,s_{n-k},0)
    \mapsto\left(k-\sum_{i=1}^{n-k}s_i,s_1,...,s_{n-k}\right)
    \in {\mathcal S}_{n,k}.
    \label{map:SnSnk}
  \end{equation}
  The map is easily confirmed with Young tableau, a collection of boxes
  arranged in left-justified boxes, with the row length in non-increasing
  order. Listing the number of boxes in each row gives a partition.  The affine
  map (\ref{map:SnSnk}) means that if we discard the rightmost column, we have
  a partition in $\mathcal{S}_{n-k}$. Because all vertices of
  $\mathcal{S}_{n-k}$ are on the faces of $P_{n-k}$, all vertices of
  $\mathcal{S}_{n,k}$ are on the faces of $({\rm New}(Z_{n,k}))$, so
  $\forall s\notin{\rm relint}({\rm New}(Z_{n,k}))$. For $2\le k<n/2$,
  the modified map
  \[
    \tilde{{\mathcal S}}_{n-k}\ni
    \left\{(s_1,...,s_{n-k},0):\sum_{i=1}^{n-k}s_i\le k\right\}
    \mapsto\left(k-\sum_{i=1}^{n-k}s_i,s_1,...,s_{n-k}\right)
    \in {\mathcal S}_{n,k}
  \]
  is one-to-one, where $\tilde{{\mathcal S}}_{n-k}$ is a collection of
  all integer partitions of $n-k$ with $\sum_{i=1}^{n-k}s_i\le k$. It can
  be shown that all vertices of $\tilde{\mathcal{S}}_{n-k}$ are on the faces
  of $P_{n-k}$ and $\forall s\notin{\rm relint}({\rm New}(Z_{n,k}))$.
\end{proof}
  
\begin{remark}\rm
  The fact that a size index never enters ${\rm relint}({\rm New}(Z_{n,k}))$
  can be seen as an observation of the integer partition: the number of
  clusters whose sizes are equal to or greater than $(n-k)/2+1$ is at most
  one. In particular, we have the vertex $e_{n-k+1}+(k-1)e_1$ and other
  vertices are $0$ in the $(n-k+1)$-th coordinate. This implies that every
  vertex of ${\mathcal S}_{n,k}$ is on the boundary of ${\rm New}(Z_{n,k})$.
\end{remark}

\begin{example}\rm
  \label{exa:k=n-2_moment}
  When $n=k+2\ge 4$ the Newton polytope is the finite open interval between
  the two possible observations $(n-4,2,0)^\top$ or $(n-3,0,1)^\top$.
  The image of the moment map is
  \begin{equation*}
    \left(\begin{array}{c}
      \eta_1\\\eta_2\\\eta_3
    \end{array}\right)=
    \left(\begin{array}{c}
      n-4\\2\\0
    \end{array}\right)
    +\left(1+\frac{n-3}{2y_1}\right)^{-1}
    \left(\begin{array}{c}
      1\\-2\\1
    \end{array}\right), \qquad y_1\in{\mathbb R}_{>0},
  \end{equation*}
  where $y_1$ is the generalized odds ratio defined in (\ref{def:odds}).
  If the observation is $(n-3,0,1)$, the likelihood is
  $(1+(n-3)/(2y_1))^{-1}$. If the observation is $(n-4,2,0)$, the likelihood is
  $(1+2y_1/(n-3))^{-1}$. The MLE does not exist for both of the cases.
\end{example}

Theorem~\ref{thm:MLE_exist} forces us to consider a sample consisting of
multiple size indices, or multiple count vectors. Assume we have an i.i.d.
sample of size $N$ and denote the multiple size indices as
$\{s^{(1)},s^{(2)},...,s^{(N)}\}$. Note that we have two sample sizes: $n$ and
$N$. The log likelihood is
\begin{equation}
  N\ell_{n,k}(s^{(\cdot)};\xi)
  =N\left\{\xi^j\bar{s}_j-\psi_{n,k}(\xi)\right\},
  \qquad \bar{s}_j:=\frac{1}{N}\sum_{i=1}^N s_j^{(i)}.
  \label{def:like_full}
\end{equation}
The MLE may exist for the sample of size $N\ge 2$, because the sample $\bar{s}$
can enter the relative interior of the Newton polytope. For the moment map
(\ref{def:mom_map}), Takayama et al.~\cite{takayama2015} established
the following theorem.

\begin{theorem}[\cite{takayama2015}]
  \label{thm:New}
  Let $A$ be a $d\times m$ homogeneous matrix with non-negative integer
  entries. If the affine dimension of the Newton polytope ${\rm New}(Z_A(b))$
  is $m-d$, then the image of the moment map $(\ref{def:mom_map})$ agrees
  with the relative interior of the Newton polytope. Moreover, the moment map
  is one-to-one.
\end{theorem}

\begin{corollary}
  \label{cor:New_B}
  For the $A$-hypergeometric distribution associated with the rational normal
  curve $(\ref{def:A_dist_MLE})$, the image of the moment map
  $(\ref{def:mom_map})$ agrees with the relative interior of the Newton
  polytope ${\rm New}(Z_{n,k})$. Moreover, the moment map is one-to-one.
\end{corollary}

Let us prepare the following lemma.

\begin{lemma}
  \label{lem:NewZnk}
  The affine dimension of the Newton polytope ${\rm New}(Z_{n,k})$ is
  $n-k-1$ for $n\ge k+2\ge 5$.
\end{lemma}

\begin{proof}
  When $n\ge k\ge n/2$, the one-to-one affine map (\ref{map:SnSnk}) implies
  that the affine dimension of the Newton polytope equals the affine dimension
  of the partition polytope $P_{n-k}$, which is $n-k-1$ by Theorem 1 in
  \cite{shlyk2005}. If $3\le k<n/2$ it is sufficient to establish that there
  exists a basis of the vector space of size index $s\in{\mathcal S}_m$,
  $m\ge 2$, which consists of vertices satisfying $\sum_{j=1}^m s_j\le 3$. In
  fact, this is true. If $m$ is even, a basis is given by
  $\{e_m,e_i+e_{m-i},2e_j+e_{m-2j}:1\le i\le(m-1)/2,1\le j \le m/2-1\}$; if
  $m$ is odd, a basis is given by
  $\{e_m,e_i+e_{m-i},2e_j+e_{m-2j}:1\le i,j\le (m-1)/2\}$.
\end{proof}

\begin{proof}[Proof of Corollary~\ref{cor:New_B}]
  If $k\ge 3$, according to Lemma~\ref{lem:NewZnk} the affine dimension of
  ${\rm New}(Z_{n,k})$ is $n-k-1$ and the condition of Theorem~\ref{thm:New}
  is satisfied. For $k=2$, the affine dimension of the Newton polytope is
  $\lfloor n/2\rfloor-1$ and Theorem~\ref{thm:New} does not work for $n\ge 5$.
  But we can prove the assertion directly as follows. If $n$ is even,
  the $A$-hypergeometric polynomial is
  $Z_{n,2}=\sum_{j=1}^{n/2-1}x_jx_{n-j}+x_{n/2}^2/2$. It can be seen that
  the Newton polytope is a pair of two simplices, one is the convex hull of
  $\{e_1,e_2,...,e_{n/2}\}$ and the other is that of
  $\{e_n,e_{n-1},...,e_{n/2}\}$. The moment map is $\eta_1=\eta_n=y_{n-3}/z$,
  $\eta_2=\eta_{n-1}=y_{n-4}/z$, $\eta_j=y_jy_{n-4-j}/z$, $3\le j\le n/2-1$
  and $\eta_n=y^2_{n/2-2}/(2z)$, where
  $z:=y_{n-3}+y_{n-4}+\sum_{j=1}^{n/2-3}y_j y_{n-4-j}+y^2_{n/2-2}/2$ and
  $y\in {\mathbb R}_{>0}^n$. This is the pair of simplices:
  $\eta_1+\eta_2+\cdots+\eta_{n/2}/2=1$ and
  $\eta_n+\eta_{n-1}+\cdots+\eta_{n/2}/2=1$ with $\eta \in {\mathbb R}_{>0}^n$.
  The one-to-one map is obvious from the explicit expressions for
  the simplices. A similar argument gives the assertion for the case of $n$
  being odd.  
\end{proof}

\begin{example}\rm
  \label{exa:k=n-2_MLE}
  This is a continuation of Example~\ref{exa:k=n-2_moment}. The image of the
  moment map agrees with the relative interior of the Newton polytope, and
  the one-to-one map is obvious. For a sample of size $N\ge 2$, let the counts
  of $(n-3,0,1)^\top$ and $(n-4,2,0)^\top$ be $N_1$ and $N_2=N-N_1$,
  respectively. Here, $N_1\sim Binom.(N,\eta_3(y_1))$. The probability that
  the MLE of the generalized odds ratio $y_1$ exist is
  $1-\eta_3^N-(1-\eta_3)^N$.
  If the MLE of $\eta$ exists, it is
  \begin{equation*}
    \bar{s}=\left(\frac{(n-3)N_1+(n-4)N_2}{N},\frac{2N_2}{N},\frac{N_1}{N}
    \right).
  \end{equation*}
  It can be seen that $\bar{s}$ is efficient. For example,
  ${\rm Var}(\bar{s}_3)=g_{33}(\xi)/N=\eta_3(1-\eta_3)/N$. The MLE of
  the generalized odds ratio $y_1$ is the unique solution of
  $\eta(y_1)=\bar{s}$, and we have $\hat{y}_1=N_1/N_2\times(n-3)/2$.
  In the limit $N\to\infty$ the MLE is consistent and the asymptotic variance
  is
  \begin{equation*}
    {\rm Var}(\log\hat{y}_1)\sim\frac{\{y_1+(n-3)/2\}^2}{Ny_1(n-3)/2}.
  \end{equation*}
  It is remarkable that the asymptotic variance increases linearly with
  sample size $n$ for large $n$.
\end{example}

It is reasonable to consider the situation in which the indeterminants $x$ are
parametrized by a few parameters, namely, a curved exponential family.
The parametrization changes the MLE dramatically; the MLE may exist even for
a sample of size $N=1$. Let $M$ be a submanifold of the Newton polytope
${\rm New}(Z_{n,k})$, where the curved exponential family is defined on it. Let
the coordinate system of $M$ be $u^a$, $a\in\{1,...,m\}$ with $m\le n-k$. We
will use the dual coordinate system to represent a point in
${\rm New}(Z_{n,k})$; a point in $M$ is parametrized as $\eta(u)$. An estimator
is a mapping from ${\rm New}(Z_{n,k})$ to $M$:
\[
  f:{\rm New}(Z_{n,k})\to M, \qquad \bar{s}\mapsto\hat{u}=f(\bar{s}).
\]
Let us call the inverse of the estimator $A(u)=f^{-1}(u)$
the estimating manifold corresponding to the point $u\in M$.
Let us prepare a new coordinate system of ${\rm New}(Z_{n,k})$ around
$\eta(u)$: A point $\eta$ is indexed by $(u,v)$, where $v$ is the index of
$\eta$ in $A(u)$, where $\eta(u)=\eta(u,0)$. The tangent space of $M$ is
spanned by $\partial_a$, while the tangent space of $A(u)$ is spanned by
$\partial_\kappa$, $m+1\le\kappa\le n-k+1$. The following theorem is
fundamental.

\begin{theorem}[\cite{amari2000}]
  \label{thm:MLE_IG}
  For a curved exponential family with submanifold $M$, an estimator $\hat{u}$
  is consistent if and only if the estimating submanifold $A$ contains point
  $\eta(u)$ as $N\to\infty$. The asymptotic covariance matrix of the estimator
  satisfies
  $\lim_{N\to\infty}N{\mathbb E}[(\hat{u}^a-u^a)(\hat{u}^b-u^b)]=\bar{g}^{ab}$,
  where $\bar{g}^{ab}:=(g_{ab}-g_{a\kappa}g^{\kappa\lambda}g_{b\lambda})^{-1}$.
  The estimator is first-order asymptotically efficient if and only if $A(u)$
  and $M$ are orthogonal.
\end{theorem}

The following corollary characterizes the MLE for the curved exponential
family.

\begin{corollary}
  \label{cor:MLE_exist_curve}
  For the log likelihood $(\ref{def:like_full})$, suppose the indeterminants
  are parametrized by a parameter $u$ such that they are restricted to
  a submanifold $M$ of the Newton polytope ${\rm New}(Z_{n,k})$.

  \begin{enumerate}

  \item[(1)] If a boundary of closure of $M$ is empty, an MLE exists. In
    particular, if $M$ is convex in the dual coordinate system, it is unique.

  \item[(2)] If a boundary of closure of $M$ is not empty, an MLE exists if and
    only if an orthogonal projection of $\bar{s}$ to $M$ is possible in terms
    of the Fisher metric $(\ref{def:metric})$.

  \item[(3)] If MLEs exist, they are consistent and first order asymptotically
    efficient as $N\to\infty$. The asymptotic covariance matrix is given by
    $g_{ab}^{-1}/N$.

  \end{enumerate}
  
\end{corollary}

\begin{proof}
  An MLE $\hat{u}$ gives a stationary point $\eta(\hat{u})$ in a manifold $M$
  that minimizes the Kullback-Leibler divergence between a sample $\bar{s}$
  and $M$. Thanks to the dually flatness of the exponential family, this is
  the orthogonal projection of the point $\bar{s}$ onto $M$ along with
  the m-geodesic, which is the straight interval that connects $\bar{s}$
  and $\eta(\hat{u})$ in the dual coordinate system. Assertion (1) is a special
  case. If a boundary of closure of $M$ is empty, the orthogonal projection is
  always possible. If $M$ is convex, by virtue of the Hilbert projection
  theorem, the point $\eta(\hat{u})$ uniquely minimizes the Kullback-Leibler
  divergence over $M$. Therefore, the MLE $\hat{u}$ is unique. For assertion
  (3), the straight interval that connects $\bar{s}$ and $\eta(\hat{u})$ is
  the estimating submanifold $A(u)$ and the MLE $\hat{u}$ is consistent.
  Because $A(u)$ is orthogonal to $M$, the MLE $\hat{u}$ is first order
  efficient. The asymptotic covariance follows from Theorem~\ref{thm:MLE_IG}.
\end{proof}

\begin{remark}\rm
  For assertion (2) of Corollary~\ref{cor:MLE_exist_curve}, the possibility of
  orthogonal projection must be checked case by case. For example, if $M$ is
  convex, the orthogonal projection is possible if and only if $\bar{s}$ is not
  in the union of normal cones of the boundary of the closure of $M$:
  \[
    \bigcup_{\mu\in{\partial{\bar{\mathcal M}}}}\{\eta\in {\rm New}(Z_{n,k}):\forall\mu^*\in{\mathcal M},\langle\eta,\mu-\mu^*\rangle\ge 0\},
  \]
  where the inner product $\langle\cdot,\cdot\rangle$ is in terms of the Fisher
  metric (\ref{def:metric}). If this condition is satisfied the orthogonal
  projection is possible and the MLE $\hat{u}$ exists uniquely.
\end{remark}

As an example of the curved exponential family, let us consider
the $A$-hypergeometric distribution (\ref{def:mGibbs}) which emerges as
the conditional distribution of the Gibbs random partition. The submanifold
$M$ is one-dimensional and parametrized by $\alpha\in(-\infty,1)$.
By the parametrization (\ref{def:egibbs}) the generalized odds ratio becomes
\begin{equation}
  y_i=\frac{2^{i+1}}{(i+2)!}
  \frac{(1-\alpha)_{i+1}}{(1-\alpha)^{i+1}}, \qquad i\in\{1,...,n-k-1\}.
  \label{def:egibbs_y}
\end{equation}
The image of the moment map $M$ is now a smooth open curve in
${\rm relint}({\rm New}(Z_{n,k}))$. One of the limit points as $\alpha\to 1$,
which is called the Fermi-Dirac limit, is $\eta=(k-1)e_1+e_{n-k+1}$.
This is a vertex of ${\rm New}(Z_{n,k})$. Here, the Fisher metric is $0$.
Another limit is $\alpha\to -\infty$, which is called the Maxwell-Boltzmann
limit. A simple calculation shows that the other limit is
\[
  \eta_i=
  \left(\begin{array}{c}n\\i\end{array}\right)\frac{S(n-i,k-1)}{S(n,k)}
\]
and the Fisher metric is
\[
  g_{ij}=
  \left(\begin{array}{c}n\\i,j\end{array}\right)\frac{S(n-i-j,k-2)}{S(n,k)}
  I_{\{n-k+2\ge i+j\}}-\eta_i\eta_j+\eta_i\delta_{i,j},
\]
where $S(n,k)$ denotes the Stirling number of the second kind. The inverse of
the $N$ times the asymptotic variance is
$g_{\alpha\alpha}=\|\partial_\alpha^2\|=g_{ij}\partial_\alpha\xi^i\partial_\alpha\xi^j$,
where $\partial_\alpha\xi^i=\sum_{j=1}^{i-1}(\alpha-j)^{-1}$, $i\ge 2$, and
$\partial_\alpha\xi^1=0$, which is the squared norm of the tangent vector
along with the curve $M$. The squared norm vanishes as $\alpha\to 1$ and
diverges as $\alpha\to-\infty$, which implies that the model is singular at
these limit points. The following examples illustrate the nature of the MLE.

\begin{example}\rm
  \label{exa:k=n-2_a}
  This is a continuation of Example~\ref{exa:k=n-2_MLE}, where $n=k+2\ge 4$.
  The Newton polytope is the finite open interval between $(n-4,2,0)^\top$
  and $(n-3,0,1)^\top$. The submanifold $M$ is a subset of the Newton polytope,
  which is the open interval between the two points
  $((n-3)(3n-8)/(3n-5),6(n-3)/(3n-5),4/(3n-5))^\top$ and $(n-3,0,1)^\top$.
  The former point corresponds to the limit $\alpha\to -\infty$, while
  the latter point corresponds to the limit $\alpha\to 1$. Here, the orthogonal
  projection is the identity map. The MLE does not exist for
  a sample of size $N=1$. For a sample of size $N\ge 2$, if $\bar{s}$ is
  within the interval, which is equivalent to $0<4N_2<3(n-3)N_1$, the MLE
  exists uniquely. The asymptotic variance
  \[
    N^{-1}g^{\alpha\alpha}=
    \frac{1}{N}\left(1+\frac{3(n-3)}{4}\frac{\alpha-1}{\alpha-2}\right)^2
    \left[\frac{2}{(\alpha-1)^2}+\frac{(n-3)}{36}
    \left\{\frac{22\alpha^2-56\alpha+43}{(\alpha-1)(\alpha-2)^3}\right\}
    \right]^{-1}
  \]
  increases linearly with sample size $n$, as for the full exponential
  family discussed in Example~\ref{exa:k=n-2_MLE}.   
\end{example}

\begin{example}\rm
  \label{exa:k=n-3_a}
  When $n=k+3\ge 6$, the dimension of the Newton polytope is two. The Newton
  polytope is the convex hull of the three vertices $(n-6,3,0,0)^\top$,
  $(n-5,1,1,0)^\top$, and $(n-4,0,0,1)^\top$. The image of the moment map is
  \[
    \left(\begin{array}{c}
      \eta_1\\\eta_2\\\eta_3\\\eta_4
    \end{array}\right)=
    \left(\begin{array}{c}
      n-6\\3\\0\\0
    \end{array}\right)
    +\frac{\frac{3!}{n-5}y_1}{1+\frac{3!}{n-5}y_1+\frac{3!}{(n-4)(n-5)}y_2}
    \left(\begin{array}{c}
      1\\-2\\1\\0
    \end{array}\right)
    +\frac{\frac{3!}{(n-4)(n-5)}y_2}
      {1+\frac{3!}{n-5}y_1+\frac{3!}{(n-4)(n-5)}y_2}
    \left(\begin{array}{c}
      2\\-3\\0\\1
    \end{array}\right).
  \]
  One of the limit points of the curve $M$ with $\alpha\to 1$ is
  $(n-4,0,0,1)^\top$, while the other limit point with $\alpha\to -\infty$ is
  \begin{equation*}
    \left(\frac{(n-4)^2}{n-2},\frac{(3n-11)(n-4)}{(n-2)(n-3)},
    \frac{4(n-4)}{(n-2)(n-3)},\frac{2}{(n-2)(n-3)}\right)^\top.
  \end{equation*}
  The latter point is in ${\rm relint}({\rm New}(Z_{n,k}))$, but in the limit
  $n\to\infty$ it tends to $(n-6,3,0,0)$, which is a vertex of
  ${\rm New}(Z_{n,k})$. Because the curve $M$ is not convex, it is not
  straightforward to check for the possibility of orthogonal projection.
  However, an analysis of the estimating equation suggests that MLEs do not
  exist for small $n$. The estimating equation can be recast into
  $f(\alpha)=0$, where
  \begin{align*}
    f(\alpha)&:=
      \{-(\bar{s}_3+3\bar{s}_4)n^2+(5\bar{s}_3+15\bar{s}_4+4)n
        -2(3\bar{s}_3+9\bar{s}_4+5)\}\alpha^3\\
      &+\{(5\bar{s}_3+13\bar{s}_4)n^2-(21\bar{s}_3+53\bar{s}_4+24)n
        +4(5\bar{s}_3+12\bar{s}_4+13)\}\alpha^2\\
      &+\{-(7\bar{s}_3+17\bar{s}_4)n^2+(19\bar{s}_3+45\bar{s}_4+44)n
        -2(3\bar{s}_3+7\bar{s}_4+35)\}\alpha\\
      &+(3\bar{s}_3+7\bar{s}_4)n^2-(3\bar{s}_3+7\bar{s}_4+24)n+12.
  \end{align*}
  The problem of MLE existence is interpreted as an elementary analytical
  problem of the existence of the solution of $f(\hat{\alpha})=0$ with
  $\partial_\alpha f(\hat{\alpha})<0$ and $\hat{\alpha}<1$. It can be shown
  that the MLE exists uniquely if and only if the coefficient of $\alpha^3$
  is negative, which is equivalent to the condition
  \begin{equation}
    \bar{s}_3+3\bar{s}_4>\frac{2(2n-5)}{(n-2)(n-3)}.
    \label{cond:proj}
  \end{equation}
  A sketch of the proof is as follows. It is obvious that no MLE exists for
  $\bar{s}=(n-4,0,0,1)^\top$. Let us assume that $\bar{s}\neq(n-4,0,0,1)^\top$.
  Then $f(1)=8(2\bar{s}_4+\bar{s}_3-2)<0$. Because of the nature of a cubic
  curve, it is obvious that the necessary condition of the existence of the
  MLE is that the coefficient of $\alpha^3$ is negative, in which case there
  may be a possibility that two MLEs of the same likelihood exist. Necessary
  conditions for the existence of the two MLEs is $\partial_\alpha f(1)<0$ and
  the solution of $\partial_\alpha f=0$ is smaller than 1. However, we can
  check that the intersection of these conditions for $s$ is empty. Therefore,
  the condition that the coefficient of $\alpha^3$ is negative is sufficient
  for the unique existence of the MLE. Let us view (\ref{cond:proj}) as being
  certainly the condition that determines the possibility for the orthogonal
  projection around $\alpha\to-\infty$. Let
  $B_{\alpha i}:=\partial_\alpha\eta_i(-\infty,0)=g_{ij}\partial_\alpha\xi^j$
  and $B_{\kappa i}:=\partial_\kappa\eta_i(-\infty,0)$,
  where $\partial_\alpha=B_{\alpha i}\partial^i$ and
  $\partial_\kappa=B_{\kappa i}\partial^i$ are the tangent vectors of
  $M$ and $A(-\infty)$ expressed in terms of basis $\{\partial^i\}$,
  respectively. Taking
  $\partial_\kappa=\delta_{\kappa, 2}(\bar{s}_i-\eta_i(u))\partial^i$,
  the condition of possibility for the orthogonal projection is  
  \begin{equation*}
    g_{\alpha 2}
    =\langle\partial_\alpha,\partial_2\rangle=B_{\alpha i}B_{2j}g^{ij}
    =\partial_\alpha\xi^j(\bar{s}_j-\eta_j(-\alpha))> 0,
  \end{equation*}
  which is equivalent to (\ref{cond:proj}). The remarkable difference from
  the case of the full exponential family is that MLE exists even for the case
  of $N=1$. In fact, it can be seen that $s=(n-5,1,1,0)$ with $n\ge 7$ has
  the MLE. If the MLE exists, the asymptotic variance with $N\to\infty$ is
  $g^{\alpha\alpha}/N\sim n(\alpha-1)^3(\alpha-2)/(4N)$ for large $n$.
  The asymptotic variance increases linearly with sample size $n$, as in
  the case of $n=k+2$ in Example~\ref{exa:k=n-2_a}. Figure 2 depicts
  the Newton polytope ${\rm New}(Z_{10,7})$ projected onto
  the $\eta_3$-$\eta_4$ plane, which is the lower triangle of the diagonal,
  and the submanifold $M$ is the curve. The estimating manifold for the case
  of $\bar{s}=(4.8,1.6,0.4,0.2)$ is shown by the arrow, and the MLE is
  $\hat{\alpha}=0.073$. The shaded region for ${\bar s}$ is the region in
  which no MLE exists, which is the normal fan at the limit point of
  $\alpha\to -\infty$.
\end{example}

\begin{remark}\rm
  Essentially the same argument as ours here provides a classical result on
  the existence of the MLE for a sample from the Dirichlet distribution
  \cite{levin1977}. The log likelihood of the symmetric $m$-variate
  Dirichlet-multinomial distribution of parameter $(-\alpha)$ is given by
  $N\{\xi^i\bar{s}_i-\log (-m\alpha)_n\}$, $\xi_i=\log x_i$, 
  $x_i=(-\alpha)_i/i!$, where a constant term is omitted. This is a curved
  exponential family. Theorem 1 of \cite{levin1977}, which was proved
  using the variation-diminishing property of the Laplace transform, says
  that the MLE exists uniquely if and only if
  \begin{equation}
    \sum_{i=1}^ni^2\bar{s}_i>n+\frac{n(n-1)}{m}
    \label{cond:MLE_DM}
  \end{equation}
  is satisfied. In our context, the assertion is as follows. The moment map for
  the full exponential family is now
  \[
    \eta_i=\frac{\sum_{k=1}^m[m]_kZ_{n-i,k-1}(x)}
      {\sum_{k=1}^m[m]_kZ_{n,k}(x)}x_i, \qquad i\in \{1,...,n\},
  \]
  and the image is the partition polytope $P_n$ instead of the Newton polytope
  ${\rm New}(Z_{n,k})$. The submanifold $M$ is parametrized by
  $\alpha\in(-\infty,0)$ and the two limit points are $\eta=e_{n}$ and
  \[
    \eta_i=\frac{(m-1)^{n-i}}{m^{n-1}}
    \left(\begin{array}{c}n\\i\end{array}\right),
  \]
  which correspond to limits of $\alpha\to 0$ and $\alpha\to -\infty$,
  respectively. The MLE exists if the size index $\bar{s}$ is outside
  the normal fan at $\alpha=-\infty$, which is equivalent to
  (\ref{cond:MLE_DM}).
\end{remark}

Before closing this section, let us summarize numerical methods for evaluating
the MLE. The discussion for the general $A$-hypergeometric distribution was
given in \cite{takayama2015}. For the full exponential family
(\ref{def:like_full}), the MLE is
\[
\hat{y}:={\rm argmax}_y f(y), \qquad
f(y)=\sum_{i=1}^{n-k-1}\bar{s}_{i+2}\log y_i-\log Z_{n,k}(y).
\]
The derivative is
\[
  \frac{\partial f}{\partial y_i}=y_i^{-1}(\bar{s}_{i+2}-\eta_{i+2}(y)),
  \qquad i\in\{1,...,n-k-1\}.
\]
Evaluate $\hat{y}$ is equivalent to finding the inverse image of the map
$\bar{s}=\eta(\hat{y})$. A simple gradient descent algorithm is as follows.

\begin{algorithm}[\cite{takayama2015}]\rm
  Set $j=0$ and small $\epsilon$.
  Provide $y^{(0)}$ and $\eta^{(0)}=\eta(y^{(0)})$.

  \begin{itemize}

  \item[(1)] End if
    \[
      \frac{\partial f^{(j)}}{\partial y_i}=
      (y_i^{(j)})^{-1}(\bar{s}_{i+2}-\eta^{(j)}_{i+2})\approx 0,
      \qquad i\in\{1,...,n-k-1\};
    \]
    
  \item[(2)] Else set
    \begin{equation*}
      y^{(j+1)}_i=y^{(j)}_i+\epsilon\frac{\partial f^{(j)}}{\partial y_i},
      \qquad \eta^{(j+1)}=\eta(y^{(j+1)}),
    \end{equation*}
    and go to (1) while incrementing $j$ by 1.

  \end{itemize}

\end{algorithm}

If we use Newton's method, which is called the natural gradient method in
information geometry, $\partial f/\partial y_i$ may be replaced with 
\[
  \sum_{j=1}^{n-k-1}(H^{-1})_{ij}(\bar{s}_{j+2}-\eta_{j+2}(y)),
\]
where
\[
  (H)_{ij}:=\frac{\partial\eta_{i+2}}{\partial y_j}
  =y_j^{-1}g_{i+2,j+2}, \qquad i,j\in\{1,...,n-k-1\}.
\]
With some tedious algebra, it can be seen that the Fisher metric $g_{ij}$
can be computed by using the Pfaffians, whose explicit form will be given
in Section~\ref{sect:comp}, and the dual coordinate:
\[
    g_{ij}=\sum_{l=1}^{n-k-i}(\tilde{P}_i^{(n)})^{-1}_{j-2,l}\eta_{l+2}
    I_{\{n-k+2\ge i+j\}}-\eta_i\eta_j+\eta_i\delta_{i,j},\qquad
    3\le i,j\le n-k+1.
\]
The symmetry of the Fisher metric $g_{ij}=g_{ji}$ is equivalent to
\[
  \sum_{l=1}^{n-k-i}(\tilde{P}^{(i)})^{-1}_{j-2,l}
  =\sum_{l=1}^{n-k-j}(\tilde{P}^{(j)})^{-1}_{i-2,l},
\]
which is the integrability condition of the Pfaffian system (see
Section~\ref{sect:comp}). Compared with the simple gradient descent, Newton's
method demands the cost of the matrix inversion.

For the curved exponential family, the algorithm needs to be modified slightly.
As an example, we consider the parametrization given in (\ref{def:egibbs_y}).
The gradient descent algorithm is now as follows.

\begin{algorithm}\rm
  Set $j=0$ and small $\epsilon$.
  Provide $\alpha^{(0)}$ and $\eta^{(0)}=\eta(\alpha^{(0)})$.

  \begin{itemize}

  \item[(1)] End if
    \[
      \frac{\partial f^{(j)}}{\partial\alpha}=
      \sum_{i=1}^{n-k-1}\frac{\partial y_i^{(j)}}{\partial\alpha}
      \frac{\partial f^{(j)}}{\partial y_i}\approx 0;
    \]
    
  \item[(2)] Else set
    \[
      \alpha^{(j+1)}=\alpha^{(j)}+\epsilon\frac{\partial f^{(j)}}
      {\partial\alpha},
      \qquad \eta^{(j+1)}=\eta(y^{(j+1)}),
    \]
    and go to (1) while incrementing $j$ by 1,  

  \end{itemize}
  where
    \[
    \frac{\partial y_i}{\partial\alpha}
    =y_i\left\{\frac{i}{1-\alpha}
           -\left(\frac{1}{2-\alpha}+\cdots+\frac{1}{i+1-\alpha}\right)
     \right\}.
    \]
\end{algorithm}
If we use Newton's method, $\partial f/\partial\alpha$ may be replaced with
$(\partial f/\partial\alpha)^{-1}\partial f/\partial\alpha$,
where
\[
\frac{\partial^2 f}{\partial\alpha^2}
=\sum_{i=1}^{n-k-1}
\left\{\frac{\partial^2y_i}{\partial\alpha^2}
\frac{\partial f^{(j)}}{\partial y_i}
+y_i^{-2}\eta_i\left(\frac{\partial y_i}{\partial\alpha}\right)^2
\right\}
-\sum_{i=1}^{n-k-1}\sum_{j=1}^{n-k-1}y_i^{-1}y_j^{-1}g_{ij}
\frac{\partial y_i}{\partial\alpha}\frac{\partial y_j}{\partial\alpha}
\]
with
\begin{align*}
  \frac{\partial^2 y_i}{\partial\alpha^2}
  =&y_i\left[\left\{\frac{i}{1-\alpha}
    -\left(\frac{1}{2-\alpha}+\cdots+\frac{1}{i+1-\alpha}\right)
    \right\}^2\right.\\
  &\left.+\frac{i}{(1-\alpha)^2}
    -\left(\frac{1}{(2-\alpha)^2}+\cdots+\frac{1}{(i+1-\alpha)^2}\right)
    \right].
\end{align*}

\begin{example}\rm
  The data sets considered are from \cite{baayen2001} and concern word
  usage of Lewis Carroll in two works, namely, \textit{Alice's Adventure in
  Wonderland $($Alice in Wonderland$)$} and \textit{Through the looking-glass
  and what Alice found there $($Through the looking-glass$)$.} An empirical
  Bayes approach is as follows. In these data, the size index $s_i$ is
  the number of word types that occur exactly $i$ times. \textit{Alice in
  Wonderland} consists of $n=26,505$ word tokens, and the number of different
  word types in the full text of $26,505$ word tokens is $k=2,651$. For
  example, a word type ``Alice'' occurs exactly 386 times and other word types
  do not occur exactly 386 times, so $s_{386}=1$. Consider application of 
  a Gibbs random partition to the data set. A Gibbs random partition is
  the marginal likelihood of a sample taken from some prior process, and has
  parameter $v$ and $x$ in (\ref{def:Gibbs}), where $x$ has the parametrization
  (\ref{def:egibbs}). Suppose we do not have interest in the parameter $v$.
  Because the number of different word types $k$ is the sufficient statistics
  for $v$, the conditional distribution is free of $v$, and is
  the $A$-hypergeometric distribution (\ref{def:mGibbs}). The (conditional)
  MLE of $\alpha$ was evaluated with the $A$-hypergeometric distribution. To
  evaluate the $A$-hypergeometric polynomials, the asymptotic approximation
  (\ref{asymp_keener}) was employed. After 56 iterations of the gradient
  descent, $\hat{\alpha}$ almost converged to $0.441$. For \textit{Through
  the looking-glass}, $n=28,767$, $k=3,085$, and $\hat{\alpha}=0.478$.
  The finding that $\hat{\alpha}>0$ implies that the Dirichlet-multinomial
  model (see Example~\ref{exa:dir_multi}) is not adequate. The poor fitting of
  the Dirichlet-multinomial model to works by William Shakespeare was pointed
  out by Keener et al.~\cite{keener1978}. Suppose we want to compare
  \textit{Alice in Wonderland} and \textit{Through the looking-glass}.
  The latter is Carroll's second story about Alice. We might hypothesize that
  Carroll benefited from his experience in writing \textit{Alice in
  Wonderland}, and that \textit{Through the looking-glass} might be
  characterized by the greater vocabulary richness. This hypothesis is
  concordant with our result, because larger $\alpha$ implies stronger tendency
  to use word type that have never occurred (see Proposition~9 of
  \cite{pitman1995}). Table 1 displays word frequency spectra of
  \textit{Alice in Wonderland} and \textit{Through the looking-glass}.  
\end{example}

\section{Computation of $A$-hypergeometric polynomials}
\label{sect:comp}

All the applications we have seen so far in
Sections~\ref{sect:test}-\ref{sect:MLE} demand practical methods to evaluate
the $A$-hypergeometric polynomials associated with the rational normal curve
at a given point of the indeterminants. This section is devoted to computational
issues. The $A$-hypergeometric polynomials satisfy a recurrence relation that
comes from the combinatorial structure of the partial Bell polynomials.
Use of the recurrence relation is a method for evaluating
the $A$-hypergeometric polynomials. Alternative algebraic methods that use
the Pfaffian system to evaluate the $A$-hypergeometric polynomials, which
are examples of the holonomic gradient method (HGM), are presented.
The performances of these algorithms are compared, and asymptotic
approximations are also discussed.

Let us discuss methods to numerically evaluate the $A$-hypergeometric
polynomial associated with the rational normal curve. We will present results
for $Z_{n,k}(x)\equiv Z_A((n-k,k)^\top;x)$, where the matrix $A$ is given in
(\ref{def:A-hyp_Bu}) with $r=n\ge k+2\ge 4$. It is straightforward to modify
the following discussion for general $A$-hypergeometric polynomials associated
with a monomial curve whose matrix $A$ has the form of (\ref{def:A-hyp_As})
and $b\in{\mathbb N}A$, by fixing some of the indeterminants to be $0$.
The cases with $n=k,k+1$ are trivial because the $A$-hypergeometric polynomials
are monomials $Z_{n,n-1}(x)=x_1^{n-2}x_2/(n-2)!$ and $Z_{n,n}(x)=x_1^n/n!$.

A method to evaluate $Z_{n,k}(x)$ is to use the recurrence relation in
Proposition~\ref{prop:aBell_u_rec}. As another method, let us discuss
applying the HGM \cite{nakayama2011,takayamaweb}. The HGM is a method for
evaluating holonomic functions numerically. For our problem, totality of
the standard monomials of the $A$-hypergeometric ideal $H_A(b)$ is given in
Proposition~\ref{prop:stand_mono}. Because the factor ring $D_{n-k+1}/I$ is
finite dimensional, we should have the following system of partial
differential equations 
\begin{equation}
  \theta_i\bullet Q_{n,k}=P^{(n,k)}_iQ_{n,k}, \qquad i\in\{1,...,n-k+1\},
  \label{def:pfaf-sys}
\end{equation}
where 
\[
  Q_{n,k}(x)=(1,\theta_3,...,\theta_{n-k+1})^\top\bullet Z_{n,k}(x).
\]
This system is called the Pfaffian system, and it represents contiguity
relations among the $A$-hypergeometric system. The first step in developing
the HGM is to obtain the Pfaffians $P^{(n,k)}$.

In principle, Pfaffians can be obtained by the Buchberger algorithm and
reductions of the standard monomials with the reduced Gr\"obner basis of
$H_A(b)$ \cite{saito2010,hibi2014}. However, such general treatment is
unrealistic because the computational cost grows rapidly with the holonomic
rank. In addition, it is non-trivial to treat the singular loci that appear
in the Pfaffians. For actual applications, explicit expressions for
the Pfaffians are inevitable for a specific solution rather than a general
one. Goto and Matsumto obtained such an expression for the $A$-hypergeometric
polynomial of type $(i+1,i+j+2)$, which appears as the normalizing constant of
the two-way contingency tables with fixed marginal sums \cite{goto2016}.
Following them, we call the vector $Q_{n,k}(x)$ the Gauss-Manin vector.

Let us consider how to obtain explicit expressions for the Pfaffians in
(\ref{def:pfaf-sys}), The first rows are immediately determined with
the annihilator (\ref{def:A-hyp_1}). That is,
\begin{align}
  &(P_1^{(n,k)})_{1,\cdot}=(2k-n,1,2,...,n-k-1),\qquad
    (P_2^{(n,k)})_{1,\cdot}=(n-k,-2,-3,...,-n+k),\nonumber\\
  &(P_i^{(n,k)})_{1,j}=\delta_{i,j+1}, \qquad 1\le j\le n-k,
  \qquad 3\le i\le n-k+1.
  \label{pfaf-1}
\end{align}
However, other rows demand some consideration. Taking derivatives of
the definition of the $A$-hypergeometric polynomial (\ref{def:A-hyp_pol}),
we have
\begin{equation}
  \theta_i\bullet Z_{n,k}(x)=x_iZ_{n-i,k-1}(x),
  \qquad 1\le i\le n-k+1.
  \label{thetaZ}
\end{equation}
Therefore, the Gauss-Manin vector becomes a simple expression, namely
\[
  Q_{n,k}(x)=(Z_{n,k},x_3 Z_{n-3,k-1},..., x_{n-k+1}Z_{k-1,k-1})^\top.
\]
Because the $A$-hypergeometric polynomial has finite terms, higher-order
differential operators provide annihilators. Using (\ref{thetaZ}), it can be
seen that the second derivative yields annihilators:
\[
  \theta_i\theta_j-\delta_{i,j}\theta_i, \qquad i+j\ge n-k+3.
\]
Finally, the recurrence relation in Proposition~\ref{prop:aBell_u_rec} yields
the following annihilators:
\begin{eqnarray}
  &&n\theta_i-\sum_{j=0}^{n-k+1-i}(j+1)\theta_i\theta_{j+1},
  \qquad j\le i\le (n-k+2)/2,\label{pfaf-rec1}\\
  &&(n-i)\theta_i-\sum_{j=0}^{n-k+1-i}(j+1)\theta_i\theta_{j+1},
  \qquad (n-k+2)/2< i\le n-k+1.
  \label{pfaf-rec2}
\end{eqnarray}
By using the annihilators (\ref{def:A-hyp_2}), the annihilators
(\ref{pfaf-rec1}) and (\ref{pfaf-rec2}) are recast into 
\begin{eqnarray*}
  &&(n-i)\theta_i+(2j-i)\frac{x_ix_{l+1}}{x_j^2}\theta_j
  -\sum_{l=0}^{n-k+1-i}(l+1)\frac{x_ix_{l+1}}{x_jx_{i+l+1-j}}
  \theta_j\theta_{i+l+1-j}, \,\, i+1\le 2j\le n-k+2,\\
  &&(n-i)\theta_i-\sum_{l=0}^{n-k+1-i}(l+1)\frac{x_ix_{l+1}}{x_jx_{i+l+1-j}}
  \theta_j\theta_{i+l+1-j}, \qquad 2j<i+1,\,2j>n-k+2,
\end{eqnarray*}
for $j \le i \le n-k+1$. Solving this system of annihilators for the second
derivatives we can obtain the Pfaffian system (\ref{def:pfaf-sys}).

\begin{lemma}
  \label{lem:pfaffian}
  The elements of the Pfaffians for the $A$-hypergeometric polynomial
  $Z_{n,k}(x)$ are, for $1\le l,m\le n-k$ and $1\le i\le n-k+1$,
  \begin{equation}
    (P_i^{(n,k)})_{l,m}=\delta_{l,1}(P_i^{(n,k)})_{1,m}
    +\delta_{l,m}\delta_{l,i-1}I_{\{i\ge 3\}}
    +(\tilde{P}_i^{(n)})^{-1}_{l-1,m-i}
    I_{\{2\le l\le n-k-i+1,i+1\le m\le n-k\}},
    \label{pfaf-2}
  \end{equation}
  where $(P_i^{(n,k)})_{1,m}$ are given in $(\ref{pfaf-1})$ and
  $\tilde{P}_i^{(n)}$ are upper triangular matrices with elements
  \begin{equation*}
    (\tilde{P}_i^{(n)})_{l,m}
    :=\frac{m-l+1}{n-i-l-1}\frac{x_{m-l+1}x_{i+l+1}}{x_{m+2}x_i},
    \qquad 1\le l\le m\le n-k-i.
  \end{equation*}
\end{lemma}

The following explicit example may help explain the discussion so far.

\begin{example}\rm
  \label{exa:k=n-2}
  Let us consider the explicit solution basis for $n=k+2\ge 4$. The holonomic
  rank is ${\rm vol}(A)={\rm rank}(H_A(b))=n-k=2$. For a weight vector
  $w=(1,0,0)$   the reduced Gr\"obner basis of the toric ideal $I_A$ is
  $\{\partial_1\partial_3-\partial_2^2\}$ and the fake exponents are
  $(n-4,2,0)\in{\mathbb N}^3$ and $(0,2n-6,4-n)$.
  The unique polynomial solutions around the origin are constant multiples of
  the $A$-hypergeometric polynomial
  \begin{equation}
    Z_{n,n-2}(x)=\frac{x_1^{n-4}x_2^2}{2(n-4)!}
    \left(1+\frac{2}{n-3}y_1\right), \qquad y_1=\frac{x_1x_3}{x_2^2},
    \label{sol_k=n-2-1}
  \end{equation}
  which is the constant multiple of the partial Bell polynomial. The other
  solution basis can be obtained by perturbing $b$ (see Section 3 of
  \cite{saito2010}). The result is 
  \begin{align}
    &Z_{n,n-2}(x)\log y_1+x_2^{2n-6}x_3^{4-n}
      \left\{
      \frac{y_1^{n-2}}{(n-2)!}
        {}_3F_2\left(\frac{3}{2},1,1;n-1,3;4y_1\right)
        \right.\nonumber\\
    &\left.-\frac{y_1}{(n-3)(n-3)!}-\frac{(n-5)!}{(2n-6)!}(-1)^n
      \sum_{i=0}^{n-5}\frac{(3-n)_i(7/2-n)_i}{(5-n)_i}\frac{(-4y_1)^i}
          {i!}\right\},
   \label{sol_k=n-2-2}
  \end{align}
  for $n\ge 5$. For the case of $n=4$, the two fake exponents degenerate
  and the result is (\ref{sol_k=n-2-2}) with the last term replaced by
  $(-4y_1)$. The Pfaffian system is obtained by the Buchberger algorithm and
  reductions of the standard monomials $\{1,\theta_3\}$ with the reduced
  Gr\"obner basis of the hypergeometric ideal $H_A(b)$, whose explicit
  expression is
  \begin{align*}
    &\{\theta_1-\theta_3-n+4,\theta_2+2\theta_3-2,x_1x_3\theta_2(\theta_2-1)
    -x_2^2\theta_1\theta_3,\\
    &(x_2^2-4x_1x_3)\theta_3^2+(4x_1x_3+(n-4)x_2^2)\theta_3-x_1x_3\theta_2\}.
  \end{align*}
  The Pfaffians for $Q_{n,n-2}(x)=(Z_{n,n-2},Z_{n-3,n-3})^\top$ are
  \begin{eqnarray}
    &&P_1^{(n,n-2)}=
    \left(
      \begin{array}{cc}n-4&1\\
      \frac{2y_1}{1-4y_1}&\frac{(10-4n)y_1}{1-4y_1}
      \end{array}
    \right),\qquad
    P_2^{(n,n-2)}=
    \left(
      \begin{array}{cc}2&-2\\ 
        \frac{-4y_1}{1-4y_1}&\frac{4y_1+2(n-3)}{1-4y_1}
      \end{array}
    \right),\nonumber\\
    &&P_3^{(n,n-2)}=
    \left(
      \begin{array}{cc}0&1\\
      \frac{2y_1}{1-4y_1}&\frac{-6y_1-(n-4)}{1-4y_1}
      \end{array}
      \right).
  \label{pfaf-g}    
  \end{eqnarray}
  The singular loci is $y_1=1/4$, which is on the boundary of the convergence
  radius of the expression (\ref{sol_k=n-2-2}). A linear combination of
  the above two solution bases satisfies the Pfaffian system
  (\ref{def:pfaf-sys}) with Pfaffians (\ref{pfaf-g}). In contrast,
  the $A$-hypergeometric polynomial (\ref{sol_k=n-2-1}) satisfies
  the Pfaffian system (\ref{def:pfaf-sys}) with Pfaffians (\ref{pfaf-2}):
  \[
  P_1^{(n,n-2)}=\left(\begin{array}{cc}
      n-4&1\\0&n-3
    \end{array}\right),\qquad
  P_2^{(n,n-2)}=\left(\begin{array}{cc}
      2&-2\\0&0
    \end{array}\right),\qquad
  P_3^{(n,n-2)}=\left(\begin{array}{cc}
      0&1\\0&1
    \end{array}\right),
  \]
  but (\ref{sol_k=n-2-2}) does not satisfy it.
\end{example}
  
Let us discuss how to evaluate the Gauss-Manin vector $Q_{n,k}(x)$ at a given
point of indeterminants $x\in{\mathbb R}^{n-k+1}_{>0}$. The original HGM
is as follows \cite{nakayama2011}. Because the difference of $Q_{n,k}(x)$ can
be approximated as
\[
  Q_{n,k}(x+h)-Q_{n,k}(x)
    \approx\sum_{i=1}^{n-k+1}\frac{h_i}{x_i}\theta_i\bullet Q_{n,k}
    =\sum_{i=1}^{n-k+1}\frac{h_i}{x_i}P_i^{(n,k)}Q_{n,k},
\]
a numerical integration method for difference equations, such as
the Runge-Kutta method, provides the numerical value. For the implementation,
the initial value at some initial point of indeterminants is needed. One method
is direct evaluation of the series at the initial point \cite{hashiguchi2013}.
However, for the computation of $Z_{n,k}(x)$, simple and exact expressions are
available at some specific points of indeterminants, which comes for known
results on the partial Bell polynomials \cite{comtet1974}. For example,
\begin{equation}
  Z_{n,k}(1_\cdot)
  =\frac{1}{k!}\left(\begin{array}{c}n-1\\k-1\end{array}\right),\qquad 
  Z_{n,k}\left(\frac{(1/2)_{\cdot-1}}{\cdot!}\right)
  =\frac{(2n-k-1)!}{2^{2(n-k)}n!(n-k)!(k-1)!}
  \label{pfaf-ini},
\end{equation}
where $1_\cdot$ represents the sequence $x_i=1$, $i\ge 1$.
To evaluate
Pfaffians numerically, the cost of taking the inverse of the upper-triangular
matrix $\tilde{P}^{(i)}$ dominates. It takes roughly $O((n-k-i)^2)$ for each
$i$ and thus at most $O((n-k)^3)$ computation is needed to evaluate all
the Pfaffians. In the iteration steps in the numerical integration, the cost
scales linearly with the number of steps. Large numbers of steps give more
accurate result, with the expense of computational cost. If the initial point
is near the point at which we wish to evaluate, better accuracy can be attained
with fewer steps.

To compute the normalizing constant of the two-way contingency tables with
fixed marginal sums, another type of HGM algorithm, which is based on
difference equations among $A$-hypergeometric polynomials, was employed
\cite{goto2016}. Following \cite{ohara2015}, we call this method the difference
HGM.

Noting the derivative (\ref{thetaZ}) for $i=1,2$, the Pfaffian system is recast
into a difference equation:
\[
  x_iQ_{n-i,k-1}=P_i^{(n,k)}Q_{n,k}, \qquad 2\le k\le n-2.
\]
If $2\le k<n/2$, it is straightforward to see that the Gauss-Manin vector
can be obtained by simple matrix multiplication:
\begin{equation}
  Q_{n,k}=x_1^{k-1}\prod_{i=0}^{k-2}\left(P_1^{(n-i,k-i)}\right)^{-1}
  Q_{n-k+1,1},
  \qquad
  (Q_{n-k+1,1})_i=(\delta_{i,1}+\delta_{i,n-k})x_{n-k+1},
  \label{dhgm-1}
\end{equation}
where the inverse of the Pfaffian $P_1^{(i,j)}$ is given as
\[
  (P^{(i,j)}_1)^{-1}
  =\frac{1}{2j-i}\left(
    \begin{array}{ccccc}
      1&-1&-2&\cdots&-(i-j-1)\\
      0&  &  &(2j-i)E_{i-j-1}&\\
    \end{array}
    \right)
    \left(
    \begin{array}{cc}
      1 & 0\\
      0 & \tilde{P}_1^{(i)}\\
    \end{array}
  \right).
\]
For $n/2\le k\le n-2$, naive application of (\ref{dhgm-1}) fails because
of the singularity in $(P_1^{(i,j)})^{-1}$. Nevertheless, the following
algorithm provides the Gauss-Manin vector.

\begin{algorithm} \rm

  Let $n/2\le k\le n-2$. Compute $Z_{2(n-k)-1,n-k-1}$, $Z_{2(n-k)-3,n-k-2}$,
  ..., $Z_{5,2}$ by using (\ref{dhgm-1}).
  $Q_{4,2}=(x_1x_3+x_2^2/2,x_1x_3)^\top$.
  
  \begin{itemize}

  \item[(1)] Set $i=2$.

  \item[(2)] Increment $i$ and compute
    \begin{equation*}
      Q_{2i,i}=\frac{1}{i}
      \left(
        \begin{array}{cccccc}
          2x_1    &x_2&-1&-2  &\cdots&-(i-2)\\
          ix_1    &0  &-2i&-3i&\cdots&-i(i-1)\\
          0       &0  &   &   &E_{i-2}&\\
        \end{array}
      \right)  
      \left(
        \begin{array}{cc}
          E_2&0\\
          0  &x_2\tilde{P}_2^{(2i)}\\
        \end{array}
      \right)
      \left(
      \begin{array}{c}
        Z_{2i-1,i-1}\\
        Q_{2i-2,i-1}
        \end{array}
      \right).
    \end{equation*}

  \item[(3)] If $i<n-k$ go to (2).

  \item[(4)] Else we have
    \[
      Q_{n,k}=x_1^{2k-n}\prod_{i=0}^{2k-n-1}
      \left(P_1^{(n-i,k-i)}\right)^{-1}Q_{2(n-k),n-k}.
    \]
    
  \end{itemize}
  
\end{algorithm}

The computation costs of the three methods for numerical evaluation of
the Gauss-Manin vector are summarized as follows. Note that the recurrence
relation in Proposition~\ref{prop:aBell_u_rec} provides the Gauss-Manin
vector as a by-product. The recurrence relation demands $O((n-k)^2k)$
computation because the cost is $O((n-k)^2)$ for each $k$. For the HGM,
the cost is $O((n-k)^3)$ times the number of iteration steps. For
the difference HGM, if $2\le k<n/2$ the cost is $O((n-k)^2k)$, while if
$n/2\le k\le n-2$ it is $O((n-k)^4+(n-k)^2(2k-n))$. Accuracy is also
an important concern, but it is difficult to give a general statement.
In the following example, comparison of the performance of these three methods
is demonstrated with a specific example. Improvements of implementation of
the HGM algorithms will be discussed elsewhere.

\begin{example}\rm
  \label{exa:GFC}
  The generalized factorial coefficient, which appeared in
  Section~\ref{sect:exc}, is $1/n!$ times the $A$-hypergeometric polynomial
  with $x_i=(1-\alpha)_{i-1}/i!$, $i=1,2,...$. For the initial values of
  the HGM, the exact expressions in (\ref{pfaf-ini}) can be used; the former
  corresponds to $\alpha=-1$ and the latter corresponds to $\alpha=1/2$.
  For the Pfaffians, we have
  \[
    (\tilde{P}^{(i)})^{-1}_{l-1,m-i}=
    (-1)^l(n-m-1)\frac{[m+1]_{l+i}}{(l+1)!i!}
    \frac{(\alpha-1)}{(m-\alpha)}\frac{(\alpha-l)_{m-i}}{(i-\alpha)_{m-i}}.
  \]
  Tables 2 and 3 display results of the numerical evaluation of
  the Gauss-Manin vector by the three methods: the recurrence relation,
  the HGM, and the difference HGM. The Runge-Kutta method was employed for
  the numerical integration in the HGM, where the integration was initiated
  from the point $\alpha=-1$ and the number of steps was 500. All computations
  were implemented in quadruple-precision floating-point arithmetic in the C
  programming language and executed by one core of a 2.66 GHz Intel Core2 Duo
  CPU P8800 processor. Table 2 gives the results for $\alpha=1/2$, which was
  chosen because we know the true values (\ref{pfaf-ini}). Roughly speaking,
  the difference HGM demands less computational cost, while the recurrence
  relation gives more accurate estimates. Assume $n-k$ is small. If $k$ is
  large, the HGM and the difference HGM demand less computational cost than
  the use of the recurrence relation; otherwise, the HGM and the difference
  HGM would demand more computational cost. The HGM and the difference HGM
  lose accuracy for large $n-k$. In particular, the HGM gave negative value
  for $n-k=30$, so we omit those results. The loss of accuracy comes from
  the fact that $\alpha=-1$ and $\alpha=1/2$ are distant from each other. In
  fact, the HGM works for evaluation at $\alpha=0.1$ and gave similar values
  to those of the recurrence relation (Table 3), although we do not know
  the true values at $\alpha=0.1$.
\end{example}

If $n-k$ is large all three methods presented above fail, in which case
asymptotic approximation is inevitable. For specific parametrization of
the indeterminants $x$, we can consider the asymptotic form. However,
Theorem~6 in \cite{takayama2015} established an asymptotic approximation for
the general $A$-hypergeometric distributions in the regime of $b=\gamma\beta$,
for some $\beta\in{\rm int}(\mathbb{R}_{\ge 0}A)$, to a Gaussian density.
The asymptotic form of the $A$-hypergeometric polynomial comes from
the normalization constant. 

\begin{theorem}[\cite{takayama2015}]
  \label{thm:IPS}
  For the $A$-hypergeometric polynomial $Z_{n,k}(x)$, there exists a unique
  $m\in{\mathbb R}^{n-k+1}_{>0}$ such that $Am=(n-k,k)^\top$, $y=m^{\bar{A}}$,
  and
  \begin{equation}
    Z_{\gamma n,\gamma k}(x)\sim
    \frac{(x^m)^\gamma}{\Gamma(\gamma m+1)}
    \frac{(2\pi\gamma)^{n-k-1}}{\det(\bar{A}M^{-1}\bar{A}^\top)^{1/2}},
    \qquad \gamma\to\infty,
    \label{asymp_ips}
  \end{equation}
  where $M={\rm diag}(m)$ and
  $\Gamma(\gamma m+1)=\prod_{i=1}^{n-k+1}\Gamma(\gamma m_i+1)$.
\end{theorem}

The derivation of $m$ requires more explanation. Suppose a count vector $c$
follows a log-affine model
\[
  \frac{x^c\exp(-1\cdot x)}{c!}, \qquad
  \log x(\theta)=\bar{A}
  \bar{A}^\top(\bar{A}\bar{A}^\top)^{-1}\log y+A^\top\theta,
\]
where the generalized odds ratio $y$ is fixed. Let $\hat{\theta}(y)$ be
the MLE of $\theta$ in the model. The MLE can be evaluated numerically with
the iterative proportional scaling (IPS) procedure, which was originally
invented for contingency tables and hierarchical models. Here,
$m=x(\hat{\theta}(y))$ is the unique solution of $Am=Ac$ and $y=m^{\bar{A}}$.
An illustrative example follows.

\begin{example}\rm
  For the case of $n=k+2\ge 4$, $\bar{A}=(1,-2,1)$ and we have
  \[
    \log m=\log x(\hat{\theta}(y_1))
    =\frac{\log y_1}{6}
     \left(
       \begin{array}{c}1\\-2\\1\end{array}
     \right)
     +\left(
       \begin{array}{c}
         \hat{\theta}_2\\\hat{\theta}_1+\hat{\theta}_2\\
         2\hat{\theta}_1+\hat{\theta}_2
       \end{array}
     \right).
  \]
  $y=m^{\bar{A}}$ is obvious. IPS solves $Am=(n-k,k)^\top$, or
  \[
  \left\{
  \begin{array}{l}
     y_1^{-1/3}e^{\hat{\theta}_1+\hat{\theta}_2}+
     2y_1^{1/6}e^{2\hat{\theta}_1+\hat{\theta}_2}=n-k,\\
     y_1^{1/6}e^{\hat{\theta}_2}+
     y_1^{-1/3}e^{\hat{\theta}_1+\hat{\theta}_2}+
     y_1^{1/6}e^{2\hat{\theta}_1+\hat{\theta}_2}=k.
  \end{array}
  \right.
  \]
  In particular, if $n=2k$, $\hat{\theta}_1=0$ and
  $\hat{\theta}_2=\log\{k y_1^{1/3}/(1+2y_1^{1/2})\}$. 
\end{example}

\begin{example}\rm
  This is a continuation of Example~\ref{exa:GFC}. The accuracy of asymptotic
  forms of the generalized factorial coefficients is examined. An asymptotic
  form has been obtained by Keener et al. \cite{keener1978}. Here, we
  reproduce the result because the expressions in \cite{keener1978} contain
  some mistakes. Let $\lambda=k/n$.
  \begin{equation}
    Z_{n,k}((1-\alpha)_{\cdot-1}/\cdot!)
    \sim\frac{(1-\alpha\lambda^*)^{n-1/2}}{\sqrt{2\pi n\sigma_*^2}k!}
    \frac{\lambda^{k+1/2}}{((-\alpha)(\lambda^*-\lambda))^{k-1/2}}.
    \label{asymp_keener}
  \end{equation}
  Here, $\lambda^*$ solves
  $\lambda=\lambda^*
  \{1-((1-\alpha\lambda^*)/(-\alpha\lambda^*))^{\alpha}\}$ and the unique
  positive solution for $\alpha<0$ and the unique negative solution for
  $\alpha>0$. Moreover,
  $\sigma^2_*=-{\rm sgn}(\alpha)(\lambda^*-\lambda)
  (\lambda+\alpha(\lambda^*-\lambda)/(1-\alpha\lambda^*))/\lambda_*^2$.
  Table 4 gives some results for the case of $n=2k$. We set
  $(k,k)=\gamma(2,2)^\top$ and $m$ was evaluated at $(2,2)^\top$. It can be
  seen that the asymptotic form (\ref{asymp_keener}) gave very good
  approximation, whereas the error of the asymptotic form (\ref{asymp_ips})
  was generally large. Of course, this is not a fair comparison, since
  (\ref{asymp_keener}) was derived for specific $A$-hypergeometric
  polynomials.  
\end{example}

\vspace{5mm}

\noindent
{\bf Acknowledgments.} The author is grateful to Professor Nobuki Takayama for
insightful comments and suggestions on the holonomic gradient methods.
The author also thanks the two referees for their careful reading of
the manuscript and suggestions for improving its presentation.

\newpage

\begin{table}
\small

\begin{tabular}{rrrrrrrrrrr}
  \hline
  $i$  & $s_i$ & $\eta_i$ && $i$ & $s_i$ & $\eta_i$ && $i$ & $s_i$ & $\eta_i$\\
  \hline
  1  & 1176 & 1282.40 && 11 & 23 & 25.04 && 21 & 6 & 9.30\\
  2  & 402  &  356.63 && 12 & 20 & 21.93 && 22 & 3 & 8.65\\
  3  & 233  &  184.47 && 13 & 34 & 19.42 && 23 & 3 & 8.07\\
  4  & 154  &  117.47 && 14 & 20 & 17.34 && 24 & 6 & 7.55\\
  5  & 99   &   83.24 && 15 & 12 & 15.60 && 25 & 9 & 7.09\\
  6  & 57   &   62.96 && 16 &  9 & 14.14 && 26 & 4 & 6.66\\
  7  & 65   &   49.76 && 17 &  9 & 12.88 && 27 & 6 & 6.28\\
  8  & 52   &   40.63 && 18 & 10 & 11.80 && 28 & 3 & 5.93\\
  9  & 32   &   33.97 && 19 &  8 & 10.85 && 29 & 6 & 5.61\\
  10 & 36   &   28.94 && 20 &  5 & 10.03 && 30 & 6 & 5.32\\
  \hline
\end{tabular}

\vspace{5mm}
\begin{tabular}{rrrrrrrrrrr}
  \hline
  $i$  & $s_i$ & $\eta_i$ && $i$ & $s_i$ & $\eta_i$ && $i$ & $s_i$ & $\eta_i$\\
  \hline
  1  & 1491 & 1579.94 && 11 & 26 & 26.71 && 21 & 7 & 9.73\\
  2  & 460  &  410.76 && 12 & 30 & 23.33 && 22 & 9 & 9.04\\
  3  & 259  &  207.59 && 13 & 22 & 20.60 && 23 & 2 & 8.42\\
  4  & 148  &  130.38 && 14 & 19 & 18.35 && 24 & 3 & 7.87\\
  5  & 113  &   91.48 && 15 & 12 & 16.48 && 25 & 1 & 7.38\\
  6  & 78   &   68.68 && 16 & 21 & 14.90 && 26 & 5 & 6.93\\
  7  & 61   &   53.97 && 17 & 12 & 13.55 && 27 & 3 & 6.53\\
  8  & 47   &   43.83 && 18 & 11 & 12.39 && 28 & 7 & 6.16\\
  9  & 28   &   36.49 && 19 & 16 & 11.38 && 29 & 5 & 5.82\\
  10 & 26   &   30.98 && 20 &  9 & 10.50 && 30 & 2 & 5.52\\
  \hline
\end{tabular}

\vspace{5mm}

\caption{Word frequency spectra of \textit{Alice in Wonderland} (top) and
  \textit{Through the looking-glass} (bottom). Entries for $i>30$ are omitted.}

\end{table}
  
\newpage
  
\begin{table}
\small
\begin{tabular}{lccccc}
\hline
$n$     & 100 & 200 & 400 & 800\\
\multicolumn{5}{l}{$n-k=10$}\\
\hline
\multicolumn{5}{l}{exact}\\
$\log Z$& $-300.737$ & $-786.291$ & $-1909.67$ & $-4447.24$\\
\hline
\multicolumn{5}{l}{recursion}\\
$\log Z$& $-300.737$ & $-786.291$ & $-1909.67$ & $-4447.24$\\
time    & $0.019$ & $0.033$ & $0.067$ & $0.141$\\
\hline
\multicolumn{5}{l}{HGM}\\
$\log Z$& $-300.735$ & $-786.291$ & $-1909.67$ & $-4447.24$\\
time    & $0.084$ & $0.092$ & $0.092$ & $0.092$\\
\hline
\multicolumn{5}{l}{difference HGM}\\
$\log Z$& $-300.737$ & $-786.291$ & $-1909.67$ & $-4447.24$\\
time    & $<0.001$ & $<0.001$ & $<0.001$ & $<0.001$\\
\hline

% mathematica gives n=50, k=40, a=0.5, -99.5431 by 853109 sec.

\multicolumn{5}{l}{$n-k=30$}\\
\hline
\multicolumn{5}{l}{exact}\\
$\log Z$& $-204.912$ & $-661.958$ & $-1757.39$ & $-4267.17$\\
\hline
\multicolumn{5}{l}{recursion}\\
$\log Z$& $-204.912$ & $-661.958$ & $-1757.39$ & $-4267.17$\\
time    & $0.116$ & $0.229$ & $0.462$ & $0.928$ \\
\hline
\multicolumn{5}{l}{difference HGM}\\
$\log Z$& $-204.857$ & $-652.683$ & $-1743.14$ & $-4250.61$\\
time    & $0.003$ & $0.004$ & $0.008$ & $0.013$\\
\hline

\end{tabular}

\vspace{5mm}

\caption{Evaluations of the $A$-hypergeometric polynomial
  $Z_{n,k}((0.5)_{\cdot-1}/\cdot!)$.}

\end{table}

\newpage

\begin{table}
\small
\begin{tabular}{lccccc}
\hline
$n$  & 100 & 200 & 400 & 800\\
\multicolumn{5}{l}{$n-k=10$}\\
\hline
\multicolumn{5}{l}{recursion}\\
$\log Z$& $-295.383$& $-780.678$ & $-1903.92$ & $-4441.43$\\
time    & $0.018$ & $0.031$ & $0.067$ & $0.140$ \\
\hline
\multicolumn{5}{l}{HGM}\\
$\log Z$& $-295.383$& $-780.678$ & $-1903.92$ & $-4441.43$\\
time    & $0.092$ & $0.092$ & $0.091$ & $0.092$ \\
\hline
\multicolumn{5}{l}{difference HGM}\\
$\log Z$& $-295.383$& $-780.678$ & $-1903.92$ & $-4441.43$\\
time    & $<0.001$ & $<0.001$ & $<0.001$ & $<0.001$ \\
\hline

\multicolumn{5}{l}{$n-k=30$}\\
\hline
\multicolumn{5}{l}{recursion}\\
$\log Z$& $-192.188$ & $-646.832$ & $-1741.03$ & $-4250.18$\\
time & $0.116$ & $0.232$ & $0.462$ & $0.930$ \\
\hline
\multicolumn{5}{l}{HGM}\\
$\log Z$& $-192.194$& $-646.892$ & $-1741.17$ & $-4250.24$ \\
time    & $4.584$ & $4.592$ & $4.587$ & $4.585$\\
\hline
\multicolumn{5}{l}{difference HGM}\\
$\log Z$& $-192.178$& $-641.643$ & $-1732.10$ & $-4239.57$ \\
time    & $0.004$ & $0.005$ & $0.009$ & $0.012$\\
\hline
\end{tabular}

\vspace{5mm}

\caption{Evaluations of the $A$-hypergeometric polynomial
  $Z_{n,k}((0.9)_{\cdot-1}/\cdot!)$.}

\newpage

\end{table}
\begin{table}
\small
\begin{tabular}{cccc}
\hline
$\alpha$  & exact & (\ref{asymp_keener}) & IPS \\
\hline
\multicolumn{3}{l}{$n=800$, $k=400$, $\gamma=200$}\\
$0.5$     & $-1796.01$ & $-1796.71$ & $-2018.27$  \\
$-1$      & $-1450.24$ & $-1450.24$ & $-1561.20$  \\
\hline
\multicolumn{3}{l}{$n=400$, $k=200$, $\gamma=100$}\\
$0.5$     & $-763.047$ & $-763.739$ & $-872.232$  \\
$-1$      & $-589.888$ & $-589.888$ & $-643.657$  \\
\hline
\multicolumn{3}{l}{$n=100$, $k=50$, $\gamma=25$} \\
$0.5$     & $-126.088$ & $-126.779$ & $-150.912$ \\
$-1$      & $-82.3871$ & $-82.3846$ & $-93.7043$ \\
\hline
\multicolumn{3}{l}{$n=40$, $k=20$, $\gamma=10$}  \\
$0.5$     & $-35.1882$ & $-35.8765$ & $-43.4749$ \\
$-1$      & $-17.3794$ & $-17.3731$ & $-20.5405$ \\
\hline

\end{tabular}

\vspace{5mm}

\caption{Asymptotic approximations of the $A$-hypergeometric polynomial
  $Z_{n,k}((1-\alpha)_{\cdot-1}/\cdot!)$.}

\end{table}

\newpage

\begin{figure}

  \begin{center}

  \begin{tabular}{cc}
  \includegraphics[width=.5\textwidth]{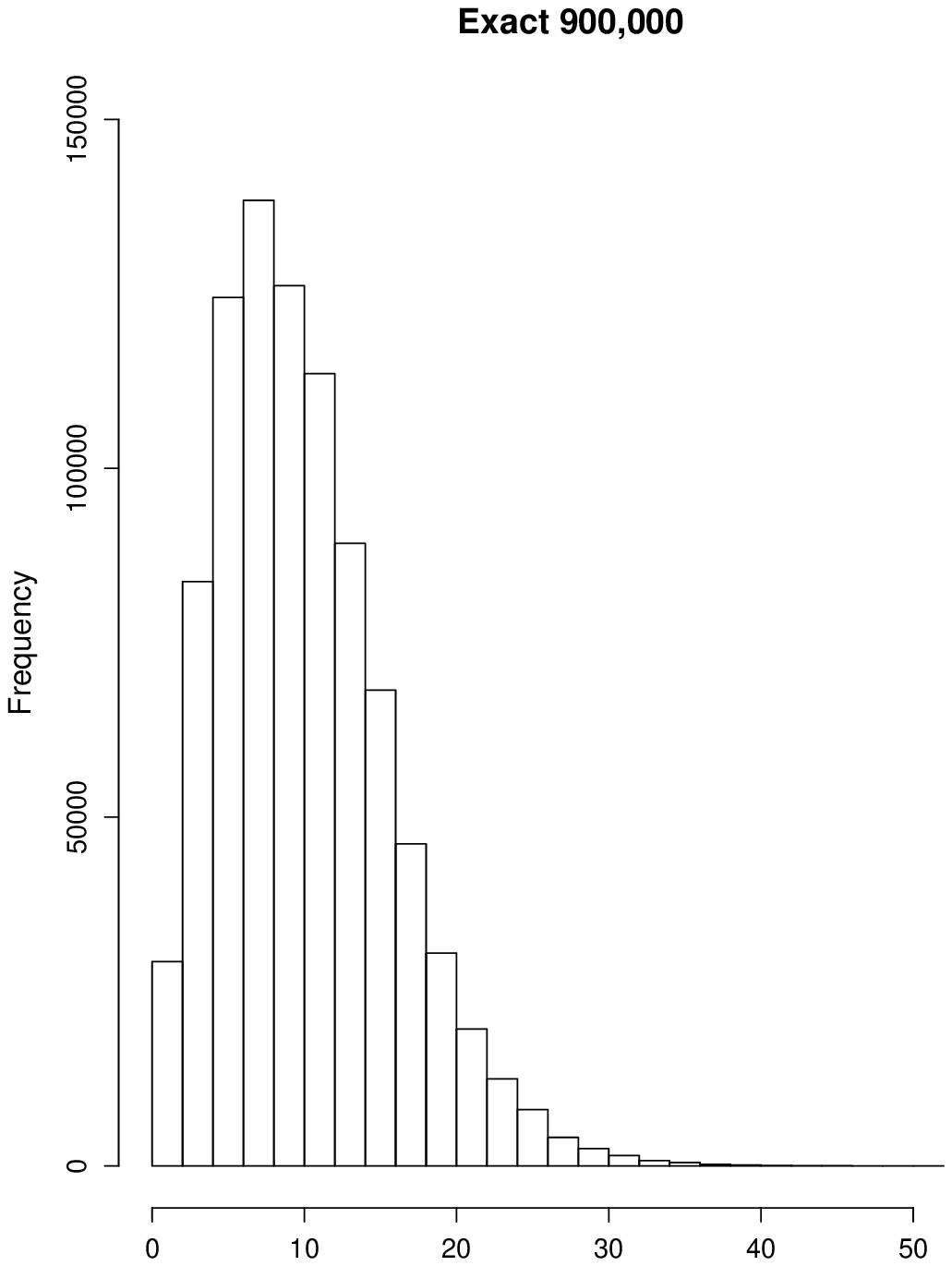}&
  \includegraphics[width=.5\textwidth]{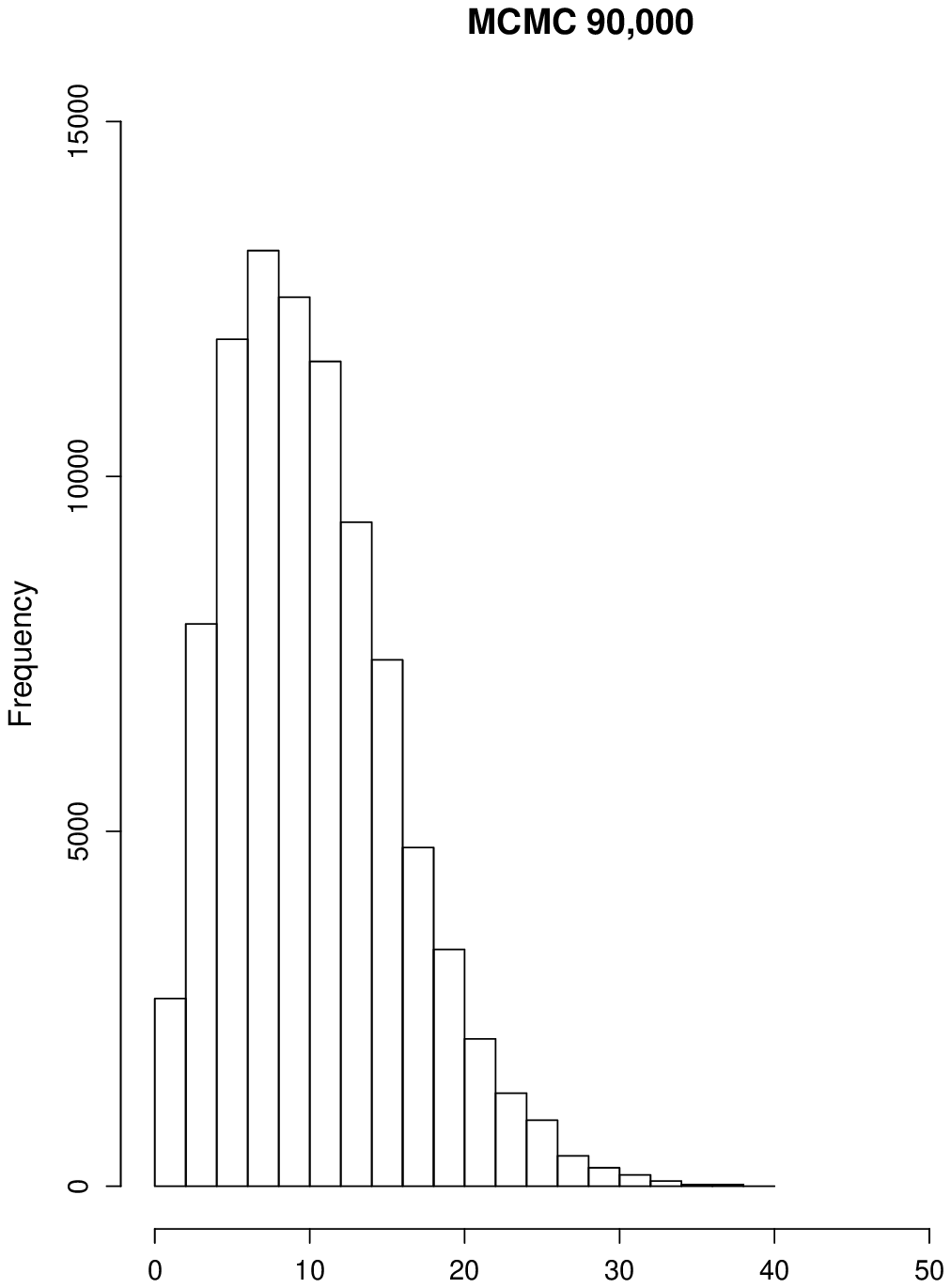}
  \end{tabular}
  
  \end{center}

  \vspace{2cm}
  
  \caption{Histograms of the $\chi^2$ statistic. Left is
    a result by the exact sampler (Algorithm~\ref{alg:exact}) with 900,000
    draws. Right is a result by the MCMC sampler with the Markov basis in
    Proposition~\ref{prop:MCMC} based on a walk of 90,000 steps (with
    the initial 10,000 steps having been discarded).}
    
\end{figure}

\newpage

\begin{figure}

  \begin{center}
    
  \includegraphics[width=\textwidth]{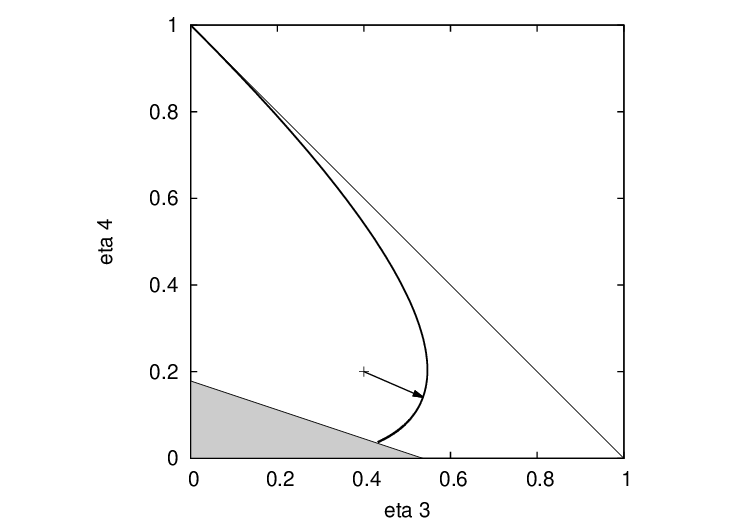}
  
  \end{center}

  \vspace{2cm}
  
\caption{The Newton polytope ${\rm New}(Z_{10,7})$ projected onto the
  $\eta_3$-$\eta_4$ plane is the lower triangle. The curve is $M$
  and the shaded region is the region of ${\bar s}$ in which no MLE exists.
  The MLE for the case of $\bar{s}=(4.8,1.6,0.4,0.2)$ is also shown as
  the arrow; the orthogonal projection from $\bar{s}$ to $M$.}

\end{figure}

\end{document}